\documentclass[11pt]{amsart}
\usepackage[utf8]{inputenc}

\usepackage{amsmath,amsthm}
\usepackage{amssymb}
\usepackage{epsfig,tikz}
\usetikzlibrary{cd}
\usepackage{graphicx}
\usepackage{asymptote}
\usepackage[all,cmtip]{xy}
\usepackage{hyperref}
\hypersetup{
    unicode=false,          
    pdftoolbar=true,        
    pdfmenubar=true,        
    pdffitwindow=false,     
    pdfstartview={FitH},    
    pdfnewwindow=true,      
    colorlinks=true,       
    linkcolor=blue,          
    citecolor=red,        
    filecolor=magenta,      
    urlcolor=cyan           
}

\usepackage[tmargin=1in,bmargin=1in,lmargin=1in,rmargin=1in]{geometry}

\newcommand{\Mbar}{\overline{M}}
\newcommand{\bdpt}{\beta}
\newcommand{\abctree}{
\begin{tikzpicture}[baseline={([yshift={-0.7*\ht\strutbox}]current bounding box.north)}]
\foreach \y in {-90, 30, 150} {
  \draw (0, 0) -- (\y:0.2); };
\fill (0, 0) circle (0.04);
\end{tikzpicture}
}

\newcommand{\abcntree}{
\begin{tikzpicture}[baseline={([yshift={-0.9*\ht\strutbox}]current bounding box.north)}]
\foreach \y in {10, 50, 90, 130, 170, 210, 250, 290, 330}
  { \draw (0, 0) -- (\y:0.2); };
\fill (0, 0) circle (0.03);
\end{tikzpicture}
}

\newcommand{\sixtree}{
\begin{tikzpicture}[baseline={([yshift={-0.9*\ht\strutbox}]current bounding box.north)}]
\foreach \y in {0, 60, 120, 180, 240, 300}
  { \draw (0, 0) -- (\y:0.2); };
\fill (0, 0) circle (0.03);
\end{tikzpicture}
}

\newcommand{\bP}{\mathbb{P}}
\newcommand{\PP}{\mathbb{P}}

\newcommand{\beq}{\begin{equation}}
\newcommand{\eeq}{\end{equation}}

\newcommand{\mb}[1]{\mathbb{#1}}
\newcommand{\mr}[1]{\operatorname{#1}}

\newcommand{\defn}{\textbf}

\newcommand{\emb}{\Omega}
\newcommand{\kapn}{|\psi_n|}
\newcommand{\kapi}{|\psi_i|}
\newcommand{\kap}{\psi}
\newcommand{\Tour}{\mathrm{Tour}}
\newcommand{\Slide}{\mathrm{Slide}}

\newcommand{\slide}{\mathrm{slide}}
\newcommand{\partition}{\sigma}

\newcommand{\hyp}{H}

\newtheorem{thm}{Theorem}
\newtheorem{lemma}[thm]{Lemma}
\newtheorem{prop}[thm]{Proposition}
\newtheorem{corollary}[thm]{Corollary}

\numberwithin{thm}{section}
\numberwithin{equation}{section}
\numberwithin{figure}{section}

\newtheorem*{FamiliesThm}{Theorem \ref{thm:families}}

\newtheorem*{StrataPsi}{Lemma \ref{thm:strata-psi}}

\newtheorem*{MainThm}{Theorem \ref{thm:main}}

\newtheorem*{MainCor}{Corollary \ref{cor:main}}

\theoremstyle{definition}

\newtheorem{example}[thm]{Example}
\newtheorem{definition}[thm]{Definition}

\newtheorem{remark}[thm]{Remark}
\newtheorem{problem}[thm]{Problem}

\newtheorem*{defHyp}{Definition \ref{def:all-hyperplanes}}

\usepackage{mathtools}
\usepackage{enumitem}

\def\multichoose#1#2{\left<\genfrac{}{}{0pt}{}{#1}{#2}\right>}

\title[Degenerations and multiplicity-free formulas on $\Mbar_{0,n}$]{Degenerations and multiplicity-free formulas \\ for products of $\psi$ and $\omega$ classes on $\overline{M}_{0,n}$}
\author{Maria Gillespie}
\thanks{Maria Gillespie was partially supported by NSF DMS award number 2054391.}
\address{Department of Mathematics, Colorado State University, Fort Collins, CO, USA} \email{maria.gillespie@colostate.edu} 

\author{Sean T. Griffin}
\thanks{Sean T. Griffin was partially supported by NSF Grant DMS-1439786 while in residence at the Institute for Computational and Experimental Research in Mathematics in Providence, RI, during the Spring 2021 semester.}
\address{Department of Mathematics, University of California Davis, Davis, CA, USA}
\email{stgriffin@ucdavis.edu}

\author{Jake Levinson}
\thanks{Jake Levinson was partially supported by an AMS Simons Travel Grant and by NSERC Discovery Grant RGPIN-2021-04169.}
\address{Department of Mathematics, Simon Fraser University, Burnaby, BC, Canada}
\email{jake\_levinson@sfu.ca}

\date{\today}

\begin{document}

\maketitle{}

\begin{abstract}

We consider products of $\psi$ classes and products of $\omega$ classes on $\Mbar_{0,n+3}$.  For each product, we construct a flat family of subschemes of $\Mbar_{0,n+3}$ whose general fiber is a complete intersection representing the product, and whose special fiber is a generically reduced union of boundary strata. Our construction is built up inductively as a sequence of one-parameter degenerations, using an explicit parametrized collection of hyperplane sections.  Combinatorially, our construction expresses each product as a positive, multiplicity-free sum of classes of boundary strata. These are given by a combinatorial algorithm on trees we call \emph{slide labeling}. As a corollary, we obtain a combinatorial formula for the $\kappa$ classes in terms of boundary strata.
 
 For degree-$n$ products of $\omega$ classes, the special fiber is a finite reduced union of (boundary) points, and its cardinality is one of the \emph{multidegrees} of the corresponding embedding $\Omega_n:\Mbar_{0,n+3}\to \PP^1\times \cdots \times \PP^n$. In the case of the product $\omega_1\cdots \omega_n$, these points exhibit a connection to permutation pattern avoidance.  Finally, we show that in certain cases, a prior interpretation of the multidegrees via \textit{tournaments} can also be obtained by degenerations.
\end{abstract}

\section{Introduction}

Let $\Mbar_{0,n+3}$ be the Deligne--Mumford moduli space \cite{deligne-mumford1969} of complex genus $0$ stable curves $C$ with $n+3$ marked points labeled by the set $\{a, b, c, 1, \ldots, n\}$. Write $\psi_i$ for the $i$-th \emph{psi class}, the first Chern class of the line bundle $\mathbb{L}_i$ whose fiber over a marked curve $(C;p_a,p_b,p_c,p_1,\ldots,p_n)\in \Mbar_{0,n}$ is the cotangent space to $C$ at the $i$-th marked point $p_i$.

We also define $\omega_i$ to be the $i$-th \emph{omega class}, the pullback of $\psi_i$ under the forgetting map $\pi:\Mbar_{0,n+3}\to \Mbar_{0,i+3}$ obtained by forgetting the marked points $p_{i+1}, \ldots, p_n$.

In this paper, we consider products in the Chow ring $A^\bullet(\Mbar_{0,n+3})$ of the form
\begin{equation} \label{eq:intersection-products}
\psi^{\mathbf{k}} := \psi_1^{k_1} \cdots \psi_n^{k_n}, \qquad \omega^{\mathbf{k}} := \omega_1^{k_1} \cdots \omega_n^{k_n},
\end{equation}
where $\mathbf{k} = (k_1, \ldots, k_n)$ is a tuple of nonnegative integers and $\sum k_i \leq n$. We introduce a family of subschemes of $\Mbar_{0,n+3}$, whose general member is a complete intersection representing $\psi^\mathbf{k}$ or $\omega^\mathbf{k}$, and whose special fiber degenerates to a generically reduced union of boundary strata.  We furthermore give a combinatorial algorithm that produces the resulting strata, in terms of the dual trees corresponding to these strata.

Our construction is by giving explicit parametrized hyperplane sections coming from the associated line bundles. The $\psi$ and $\omega$ classes give rise to two natural projective maps from $\Mbar_{0,n+3}$:
\begin{align} \label{eq:psi-n}
    \Psi_n=|\psi_1|\times \cdots \times |\psi_n|: \Mbar_{0, n+3} &\rightarrow \mathbb{P}^n \times \mathbb{P}^n \times \cdots \times \mathbb{P}^n, \\
    \label{eq:omega-n}
    \emb_n=|\omega_1|\times \cdots \times |\omega_n|: \Mbar_{0, n+3} &\hookrightarrow \mathbb{P}^1 \times \mathbb{P}^2 \times \cdots \times \mathbb{P}^n.
\end{align}

The first map is the combined or \emph{total Kapranov map} given by the psi classes, while the second map, sometimes called the \emph{iterated} Kapranov map (see \cite{CGM, GGL-tournaments,KeT, MonRan}), is an embedding and is given by the omega classes. Hyperplane sections of these maps represent the intersection products \eqref{eq:intersection-products} in $A^\bullet(\Mbar_{0,n+3})$ above.

When $\sum k_i = n$, it is well-known that the product of psi classes $\psi^\mathbf{k}$ is the multinomial coefficient $\binom{n}{k_1, \ldots, k_n}$ times the class of a point. The product of omega classes $\omega^{\mathbf{k}}$ is the so-called \emph{asymmetric multinomial coefficient} $\multichoose{n}{k_1, \ldots, k_n}$ times the class of a point \cite{CGM, GGL-tournaments}.

When $\sum k_i < n$, the products $\psi^{\mathbf{k}}$ and $\omega^{\mathbf{k}}$ represent positive-dimensional cycle classes, and by standard formulas they can be expressed as products of sums of boundary strata of $\Mbar_{0,n+3}$. In particular, using the notation $D(A|A^c)$ for the boundary divisor in which marked points $A$ are separated by $A^c$ by a node, two standard formulas for psi classes and boundary strata are
\begin{align}
    \psi_i &= \sum_\star D(i,\star \mid j,k,\star^c), \label{eq:psi-boundary} \\
    D(A \mid A^c)^2 &= -D(A \mid A^c) \bigg(\sum_\star D(a_1,a_2,\star^c \mid \star, A^c) + \sum_\star D(A,\star \mid \star^c, b_1,b_2) \bigg), \label{eq:boundary-self-intersection}
\end{align}
where in each summation, the two specified marked points ($j, k$ in the first sum, $a_1, a_2 \in A$ in the second, $b_1, b_2 \in A^c$ in the last) are arbitrary and fixed, and $\star$ ranges over all nonempty subsets of the unspecified marked points.  One can repeatedly use these formulas to expand any product of $\psi$ classes in terms of boundary divisors, but the resulting possible expressions are not unique, and it is unclear if any such expressions are actually achievable as fundamental classes of complete intersections of $\Mbar_{0,n+3}$ by hyperplanes.  Moreover, many such expansions result in alternating sums or terms with multiplicity (see Example \ref{ex:psi1psi2}), despite the fact that these products are necessarily effective and, as we will show, can be represented by generically reduced unions of boundary strata. Related work on products of psi classes includes \cite{hahn2021intersecting, kerber2009psi,smyth2018intersections}.

Our approach is as follows. For each $\mathbf{k}$, we introduce a parametrized hyperplane intersection $V^{\psi}(\mathbf{k};\vec{t})$ for $\psi^{\mathbf{k}}$ (respectively, $V^{\omega}(\mathbf{k};\vec{t})$ for $\omega^{\mathbf{k}}$) on $\Mbar_{0,n+3}$ in a tuple of parameters $\vec{t}$. We show that under a specific limit $\vec{t} \to \vec{0}$, the resulting vanishing locus on $\Mbar_{0,n+3}$ degenerates into a generically reduced union of boundary strata (Theorem \ref{thm:main}). In fact, these strata may be obtained by two closely-related combinatorial rules we call ($\psi$- and $\omega$-) \emph{slide labelings} of trees (Theorem \ref{thm:labelings}).  As a corollary, we obtain combinatorial formulas in $A^\bullet(\Mbar_{0,n+3})$ for the products $\psi^{\mathbf{k}}$ and $\omega^{\mathbf{k}}$ as positive, muliplicity-free sums of boundary strata, which moreover arise as limits of complete intersections.  A complete example of our construction, for the product $\psi_1 \psi_2$, is given in Example \ref{ex:degeneration}.

\subsection{Degenerations and slide rules}

For each $i=1, \ldots, n$, let $|\kap_i| : \Mbar_{0,n+3} \to \bP^n$ be the $i$-th Kapranov map. Let $|\omega_i| : \Mbar_{0,n+3} \to \bP^i$ be the $i$-th reduced Kapranov map, that is,
\[|\omega_i| : \Mbar_{0,n+3} \xrightarrow{\ \pi \ } \Mbar_{0,i+3} \xrightarrow{\ |\psi_i|\ } \mathbb{P}^i.\]
We give $\bP^n$ projective coordinates $[z_b:z_c:z_1:\cdots:\widehat{z_i}:\cdots:z_n]$ (where $\widehat{z_i}$ indicates that $z_i$ is omitted) and $\PP^i$ the coordinates $[w_b : w_c : w_1 : \cdots : w_{i-1}]$. Here, the hyperplane $z_j = 0$ pulls back to the union of divisors $\bigcup D(i\star|aj\star^c)$, and $w_j = 0$ is the pullback of such a hyperplane under the forgetting map $\pi$. (See Section \ref{sec:background} for background on the Kapranov map and these conventions.)

Let $t$ be a parameter. We consider the following moving hyperplane equations for $\psi_i$ and $\omega_i$.

\begin{definition}[Moving hyperplanes for $\psi_i$ and $\omega_i$] \label{def:hyperplanes}
We define the hyperplane loci
\begin{align}
\hyp_{i}^\psi(t) &= \mathbb{V}(z_b + t z_c + t^2 z_1 + \cdots + t^{i}z_{i-1} + t^{i+1}z_{i+1} + \cdots + t^n z_n) \subseteq \mathbb{P}^n, \\
\hyp_i^\omega(t) &= \mathbb{V}(w_b + t w_c + t^2 w_1 + \cdots + t^{i}w_{i-1}) \subseteq \mathbb{P}^i.
\end{align}
\end{definition}
Our construction relies on the key fact that, for $t\neq 0$, the hyperplane $\hyp_i^\psi(t)$ in $\mathbb{P}^n$ is transverse to \emph{every} boundary stratum of $\Mbar_{0,n+3}$ of every dimension. Moreover, the limiting intersection as $t \to 0$ is always a reduced union of boundary strata, which we describe by a uniform combinatorial rule. Below, we write $X_T$ for the stratum indexed by the stable tree $T$ and $\slide_i(T)$ for a set of trees defined combinatorially in Definition \ref{def:slide-rule} via \textit{slide rules}.

\begin{lemma}\label{thm:strata-psi}
  Let $T$ be a stable tree.
  Let $V_i(t)=|\psi_i|^{-1}(\hyp_i^\psi(t))$ in $\Mbar_{0,n+3}$. Then the limiting fiber is given by
  $$\lim_{t\to 0}V_i(t)\cap X_T=\bigcup_{T'\in\slide_i(T)} X_{T'}$$
  and is reduced.
\end{lemma}

For any fixed tree $T$, the right hand side above can instead be obtained by intersecting $X_T$ with a hyperplane of the form $z_j = 0$, though the particular $z_j$ depends on $T$. Intersections of the form $X_T\cap \{z_j=0\}$ are well-known and may be derived from \eqref{eq:psi-boundary}. The novelty here is the use of a single moving hyperplane for all strata $X_T$, which moreover has the following useful property.

\begin{lemma}[Injectivity]
If $T \ne T'$, the sets of trees $\slide_i(T)$ and $\slide_i(T')$ are disjoint.
\end{lemma}

This lemma leads directly to the generic reducedness statement in Theorem \ref{thm:main} below.

We now define vanishing loci $V^{\psi}(\mathbf{k};\vec{t})$ and $V^{\omega}(\mathbf{k};\vec{t})$ as intersections with, for each $i$, $k_i$ hyperplanes $\hyp_i^\psi(t)$ or $\hyp_i^{\omega}(t)$ (Definition \ref{def:hyperplanes}), with independent parameters.

\begin{definition}\label{def:all-hyperplanes}
Let $\mathbf{k}=(k_1,\ldots,k_n)$ be a weak composition.  Let $\vec{t}=(t_{i,j})$ for $1\le i\le n$ and $1\le j \le k_i$ be a tuple of complex parameters. We denote the subschemes cut out in $\Mbar_{0,n+3}$ by the hyperplanes $\hyp_i^\psi(t_{i,j})$ and $\hyp_i^\omega(t_{i,j})$ as
\begin{align}
V^\psi(\mathbf{k};\vec{t})
= \bigcap_{i=1}^n \bigcap_{j=1}^{k_i} \Psi_n^{-1}(\hyp_i^\psi(t_{i,j})),\\
V^\omega(\mathbf{k};\vec{t}) 
= \bigcap_{i=1}^n \bigcap_{j=1}^{k_i} \Omega_n^{-1}(\hyp_i^\omega(t_{i,j})),
\end{align}
where $\Psi_n$ is the total Kapranov map and $\emb_n$ is the iterated Kapranov embedding.
\end{definition}

Our main result is as follows.  There are combinatorially-defined sets of boundary strata, denoted by $\Slide^\psi(\mathbf{k})$ and $\Slide^{\omega}(\mathbf{k})$ (see Definitions \ref{def:slide-psi}--\ref{def:slide-omega}) that give a rule for the limiting intersections of hyperplanes in Definition \ref{def:all-hyperplanes}, with respect to a specific limit.

\begin{thm}\label{thm:main}
Let $\mathbf{k}$ be a weak composition and let $\vec{t} = (t_{i, j})$ for $1 \leq i \leq n$ and $1 \leq j \leq k_i$ be complex parameters. Let $\displaystyle{\lim_{\vec{t} \to \vec{0}}}$ denote the iterated limit
\[\lim_{\vec{t} \to \vec{0}} \big( {-} \big) :=
\lim_{t_{n, k_n} \to 0} \cdots \lim_{t_{n, 1} \to 0} \cdots \cdots
\lim_{t_{2, k_2} \to 0} \cdots \lim_{t_{2, 1} \to 0}\ 
\lim_{t_{1, k_1} \to 0} \cdots \lim_{t_{1, 1} \to 0} \big( {-} \big).\]
(The $i$-th block is empty if $k_i=0$, and $\lim$ denotes the flat limit.) Then we have set theoretically
\begin{equation}
    \lim_{\vec{t} \to \vec{0}} V^\psi(\mathbf{k};\vec{t}) = \bigcup_{T \in \Slide^\psi(\mathbf{k})} X_T \quad \text{ and } \quad \lim_{\vec{t} \to \vec{0}} V^\omega(\mathbf{k};\vec{t}) = \bigcup_{T \in \Slide^\omega(\mathbf{k})} X_T.
    \end{equation}
Moreover, each boundary stratum $X_T$ appearing in the union is an irreducible component and is generically reduced in the limit.
\end{thm}

As a consequence, we obtain:

\begin{corollary}
\label{cor:main}
Let $\mathbf{k}$ be a weak composition. Then in $A^\bullet(\Mbar_{0,n+3})$ we have
\begin{equation}
\psi^\mathbf{k} = \sum_{T \in \Slide^\psi(\mathbf{k})} [X_T], \qquad 
\omega^\mathbf{k} = \sum_{T \in \Slide^\omega(\mathbf{k})} [X_T].
\end{equation}
\end{corollary}

\begin{example}[A degeneration for $\psi_1\psi_2$] \label{ex:degeneration} Consider the product $\psi_1\psi_2$ on $\Mbar_{0,\{a,b,c,1,2\}}$.  Recall that $\Mbar_{0,\{a,b,c,1,2\}}$ embeds into $\PP^2\times \PP^2$ via $|\psi_1|$ and $|\psi_2|$; we coordinatize $\PP^2\times \PP^2$ as \[[x_b:x_c:x_2]\times [y_b:y_c:y_1].\]
The two hyperplane families in $\PP^2\times \PP^2$ that we will introduce, corresponding to $\psi_1$ and $\psi_2$ in the product, are \[x_b+tx_c+t^2x_2=0 \hspace{1cm}\text{and} \hspace{1cm} y_b+sy_c+s^2y_1=0 \] for parameters $t,s\in \mathbb{C}$.

We first take $t\to 0$, which gives the equation $x_b=0$.  Geometrically, the set of curves in $\Mbar_{0,\{a,b,c,1,2\}}$ that have coordinate $x_b=0$ are precisely those for which the marked point $1$ is separated from $a$ and $b$ by a node, which is the union of the three boundary strata $D(ab|c12)$, $D(abc|12)$, and $D(ab2|1c)$.  (This is a special case of the formula for $\psi_1$ given by Equation \eqref{eq:psi-boundary}.)  

In the second copy of $\PP^2$ in $\PP^2\times \PP^2$, these three boundary strata are precisely the set of curves whose coordinates satisfy either $y_b=0$ or $y_c=y_1$, which we may visualize via Figure \ref{fig:degeneration} as the two boldface blue lines in $\PP^2$.  Then the equation $y_b+sy_c+s^2y_1=0$, drawn as a dashed line in Figure \ref{fig:degeneration}, intersects these strata at two points and approaches the horizontal blue line $y_b=0$ as $s\to 0$. Note that, on the stratum where $y_b = 0$, the equation $y_b + s y_c + s^2 y_1 = 0$ yields the condition $y_c = 0$ as $s \to 0$, since $y_c$ is effectively the leading term.

The two intersection points approach the two boundary points with coordinates $[0:1:0]\times [0:1:1]$ and $[0:0:1]\times [0:0:1]$, shown at right in Figure \ref{fig:degeneration}. These boundary points may also be represented by their \defn{dual trees}:
\begin{center}
\includegraphics{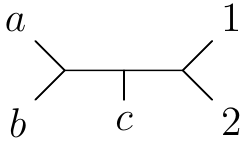}
\hspace{3cm}\includegraphics{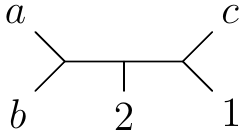}
\end{center}

\begin{figure}
    \centering
    \includegraphics[width=8cm]{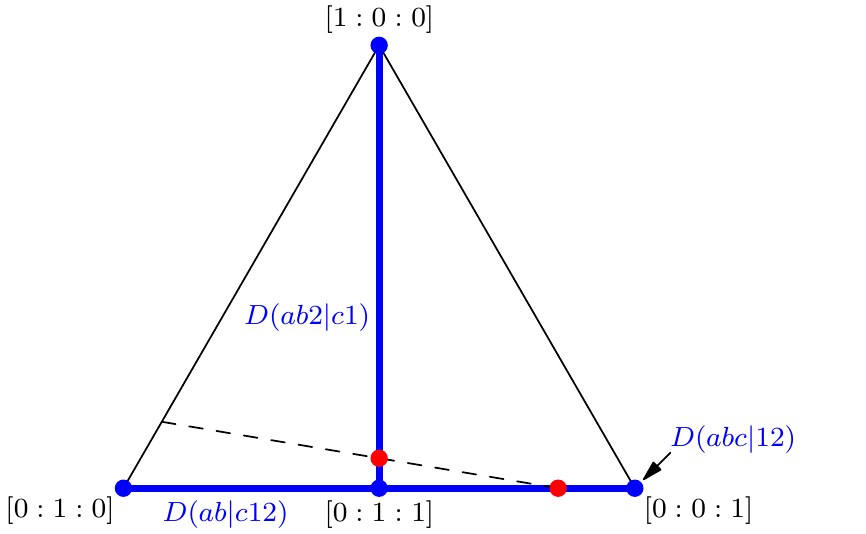}\hspace{1cm}\includegraphics{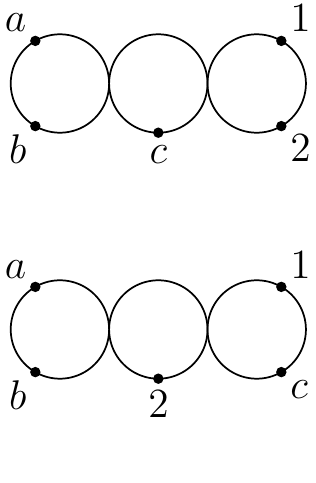}
    \caption{At left, the equation $y_b+sy_c+s^2y_1=0$ shown as a dashed line in $\PP^2$ for a small parameter $s\approx 0$.  It intersects the $\psi_1$ boundary strata, shown in boldface blue, at two points. As $s\to 0$ the two red points of intersection approach the boundary points with $y$-coordinates $[0:1:1]$ and $[0:0:1]$, drawn at right.}
    \label{fig:degeneration}
\end{figure}
\end{example}

Our choice of hyperplanes and the associated combinatorial algorithm always lead to a set of distinct trees for any product of $\psi$ or $\omega$ classes, which is not readily achieved by other known methods for calculating such products, as illustrated by the following example.
\begin{example}\label{ex:psi1psi2}
  We may calculate $\psi_1\psi_2$ directly (but without an explicit realization via hyperplanes) as follows.  By Equation \eqref{eq:psi-boundary}, we have \[
    \psi_1\psi_2=\left(D(ab|c12)+D(abc|12)+D(ab2|c1)\right)\cdot \psi_2.
\]
Expanding out the product on the right hand side, we may think of the first term as intersecting the stratum $D(ab|c12)$ with the $\psi_2$ class restricted to the component containing the marked point $2$.  Choosing $j=1$ and $k=c$ in Equation \eqref{eq:psi-boundary}, we see that this intersection gives the boundary point corresponding to the second tree in Example \ref{ex:degeneration} above.  The middle term vanishes, and for the third term, if we separate $2$ from $j=a$ and $k=b$, we again obtain the same tree as before.  Thus we find again that $\psi_1\psi_2$ is twice the class of a point, but the same tree occurs with multiplicity two in this calculation. 

Of course, all points on $\Mbar_{0,n+3}$ are rationally equivalent. However, the same issue arises for calculating products in positive dimension (even $\psi_1 \psi_2$ on $\Mbar_{0,6}$), where boundary strata are not all equivalent.
\end{example}

For further examples, see Examples \ref{ex:slide-psi} and \ref{ex:slide-omega} for the products $\psi_1\psi_3^2$ and $\omega_1\omega_3^2$, respectively.

\subsection{Application to kappa classes}

Our results and approach also yield positive boundary class formulas for the kappa classes $\kappa_i$ and generalized kappa classes, answering a question of Cavalieri \cite[p.\ 38]{cavalieri2016}. We recall that $\kappa_i$ is defined by pushforward:
\begin{equation} \label{eqn:kappa-def}
\kappa_i := (\pi_{n+1})_*(\psi_{n+1}^{i+1}) \text{ for } i \geq 0,
\end{equation}
where $\pi_{n+1}$ is the forgetting map that forgets the marked point $n+1$. The kappa classes are of particular interest in higher genus, where they are used in defining the tautological ring of $\Mbar_{g,n}$ \cite{VakilNotices}.

Below, we write $v_i$ to denote the internal vertex of a tree to which leaf edge $i$ is attached.  We write $\deg(v_i)$ for the degree of the vertex $v_i$.

\begin{definition}
Let $K(n;i) \subseteq \Slide^{\psi}(0^n,i+1)$ be the subset of trees $T$ in which $\deg(v_{n+1}) = 3$.
\end{definition}

\begin{thm}\label{thm:kappa}
On $\Mbar_{0,\{a,b,c,1,\ldots,n\}}$, we have 
\[\kappa_i=\sum_{T\in K(n;i)}[X_{\pi_{n+1}(T)}].\]
\end{thm}

The generalized kappa classes are defined similarly as iterated pushforwards: for $n \geq 3$ and a weak composition $\mathbf{r} = (r_1, \ldots, r_m)$, we define
\begin{equation}\label{eq:generalized-kappa-def}
    R_{n;\mathbf{r}} := (\pi_{n+1, \ldots, n+m})_*(\psi_{n+1}^{r_1} \cdots \psi_{n+m}^{r_m}),
\end{equation}
where $\pi_{n+1, \ldots, n+m}$ is the iterated forgetting map.

\begin{definition}
Let $R(n;\mathbf{r}) \subseteq \Slide^\psi(0^n, r_1, \ldots, r_m)$ be the subset of trees $T$ such that, for each $j = n+1, \ldots, n+m$, the tree $\pi_{j+1, \ldots, n+m}(T)$ has $\deg(v_j) = 3$.
\end{definition}

\begin{thm}\label{thm:generalized-kappa}
On $\Mbar_{0,\{a,b,c,1,\ldots,n\}}$, we have 
\[R_{n;\mathbf{r}} =\sum_{T\in R(n;\mathbf{r})}[X_{\pi_{n+1, \ldots, n+m}(T)}].\]
\end{thm}

Note that this sum is not, and likely cannot be, multiplicity-free (see Corollary \ref{cor:multiplicity} and Problem \ref{prob:kappa}).

\subsection{Multidegrees and application to tournaments} \label{sec:multidegrees}

When $\sum k_i = n$, the integers $\deg(\psi^\mathbf{k})$ and $\deg(\omega^\mathbf{k})$ are also called the \emph{multidegrees} of the maps $\Psi_n$ and $\Omega_n$, written $\deg_{\mathbf{k}}(\Psi_n)$ and $\deg_{\mathbf{k}}(\Omega_n)$.
They are the numbers of intersection points of the image of $\Mbar_{0,n+3}$ with $n$ general hyperplanes from the products of projective spaces \eqref{eq:psi-n} and \eqref{eq:omega-n}, taking $k_i$ hyperplanes from the $i$-th factor, for each $i$. Thus, a key special case of Corollary \ref{cor:main} is the following enumerative statement.

\begin{corollary}\label{cor:multidegrees}
If $k_1+\cdots+k_n=n$, we have 
\begin{align}
    \deg_{\mathbf{k}}(\Psi_n)&=\int_{\Mbar_{0,n+3}} \psi^{\mathbf{k}}=|\Slide^{\psi}(\mathbf{k})| \label{eq:psi-multidegree} \\
    \deg_{\mathbf{k}}(\Omega_n)&=\int_{\Mbar_{0,n+3}} \omega^{\mathbf{k}}=|\Slide^{\omega}(\mathbf{k})|. \label{eq:omega-multidegree}
\end{align}
\end{corollary}
It is well known that $\deg_{\mathbf{k}}(\Psi_n)$ is given by the multinomial coefficient $\binom{n}{k_1,\ldots,k_n}$ (see e.g. \cite{cavalieri2016}), so \eqref{eq:psi-multidegree} shows that this is the number of trivalent trees in $\Slide^{\psi}(\mathbf{k})$. The integers 
$$\multichoose{n}{k_1,\ldots,k_n}:=\deg_{\mathbf{k}}(\Omega_n)$$
are called the \textit{asymmetric multinomial coefficients}. A recursive formula for them was previously given in \cite{CGM}, as well as a combinatorial interpretation via parking functions.  In \cite{GGL-tournaments}, it was also shown that a different set of boundary points called $\Tour(\mathbf{k})$ also enumerates the multidegrees $\deg_{\mathbf{k}}(\Omega_n)$.  These points are defined combinatorially via an algorithm called a \textit{lazy tournament}, and we will recall the definition in Section \ref{sec:tournament-hyperplanes} below.  

The recursions underlying these prior enumerative results --- the \emph{string equation} for $\psi^\mathbf{k}$ and the \emph{asymmetric string equation} for $\omega^\mathbf{k}$ --- relate them via forgetting maps to multidegrees with one fewer marked point. The slide rule introduced in this paper, by contrast, builds up $\psi^\mathbf{k}$ and $\omega^\mathbf{k}$ from products with one fewer factor (i.e. positive-dimensional cycle classes), but the same number of marked points. These recursions seem to be entirely different, and we do not know a combinatorial analog of the (ordinary or asymmetric) string equation for the sets $\Slide^\psi(\mathbf{k})$ or $\Slide^\omega(\mathbf{k})$; it would be interesting to find one.

Along these lines, we ask whether the tournament points $\Tour(\bf k)$ may similarly be realized as limiting intersections with hyperplanes.  Our main result in this direction is that it is possible for the following families of tuples $\mathbf{k}$.

\begin{thm}\label{thm:families}
 
 Suppose the tuple $\mathbf{k}=(k_1,\ldots,k_n)$ is of one of the following forms:
 \begin{itemize}
     \item $(0,0,\ldots,0,0,n)$,
     \item $(0,0,\ldots,0,1,n-1)$,
     \item $(0,0,\ldots,0,n-1,1)$, or
     \item $(0,0,2,2)$.
 \end{itemize}
Then there exists an explicitly constructed set of hyperplanes in $\PP^1\times \cdots \times \PP^n$, with $k_i$ of them from $\PP^i$ for each $i$, such that their intersection locus $V^{\mathrm{tour}}(\mathbf{k},\vec{t})$ in $\Mbar_{0,n+3}$, pulled back under $\Omega_n$, satisfies 
\begin{equation} \label{eq:limit-tour}
\lim_{\vec{t}\to \vec{0}} V^{\mathrm{tour}}(\mathbf{k};\vec{t})=\Tour(\mathbf{k}).
\end{equation}
Moreover, given any set of hyperplanes satisfying \eqref{eq:limit-tour} for $\mathbf{k} = (k_1,\ldots,k_n)$, there exists such a set for $(k_1,\ldots,k_{n-1},0,k_{n}+1)$.
\end{thm}

\subsection{Outline of paper}

The paper is organized as follows.  We provide necessary background and notation in Section \ref{sec:background}.  In Section \ref{sec:slide-rules} we define the slide rules and give some combinatorial properties of the resulting trees.  In Section \ref{sec:limiting-hyperplanes} we prove the main theorems on degenerations, namely Theorems \ref{thm:strata-psi} and \ref{thm:main} and Corollary \ref{cor:main}, and we also prove Theorem \ref{thm:kappa}.  In Section \ref{sec:tournament-hyperplanes} we prove Theorem \ref{thm:families}, and we conclude with some further combinatorial and geometric observations in Section \ref{sec:conclusion}, including an interesting pattern avoidance condition that arises in the trees $\Slide^{\omega}(1,1,1,\ldots,1)$.

\subsection{Acknowledgments}

We thank Vance Blankers, Renzo Cavalieri, and Mark Shoemaker for several helpful discussions pertaining to this work. 

\section{Background}\label{sec:background}

We now provide some geometric and combinatorial background needed to state and prove our results.

\subsection{Structure of \texorpdfstring{$\Mbar_{0,S}$}{} and trivalent trees}\label{sec:background-trees}

Throughout, we let $S = \{a,b,c,1,2,\ldots, n\}$. A point of $\Mbar_{0,S}$ consists of an (isomorphism class of a) genus $0$ curve $C$ with at most nodal singularities and marked points labeled by the elements of $S$, such that each irreducible component has at least three \textbf{special points}, defined as marked points or nodes.  In this paper, we draw the irreducible $\PP^1$ components as circles, as in Figure \ref{fig:strata}.  The \textbf{dual tree} of a point in $\Mbar_{0,S}$ is the leaf-labeled tree formed by drawing a vertex in the center of each $\PP^1$ circle and then connecting this vertex to each marked point on its circle and each vertex on an adjacent circle connected by a node.  The dual tree is guaranteed to be a tree since the curve has genus $0$.

\begin{figure}
    \centering
    \includegraphics{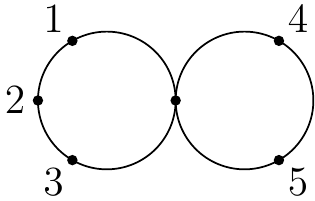}\hspace{1cm} \includegraphics{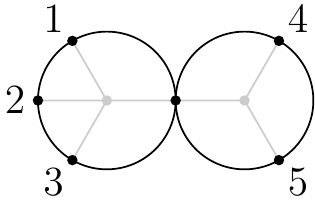}\hspace{1cm} \includegraphics{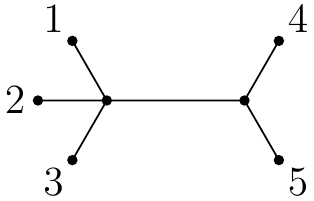}
    \caption{At left, a stable curve in $\Mbar_{0,5}$, in which each circle represents a copy of $\PP^1$.  At center, we form the dual tree of the curve shown at right.  The tree also represents the dimension-$1$ boundary stratum consisting of the closure of the set of all stable curves in which $1,2,3$ are on one component and $4,5$ are on another.}
    \label{fig:strata}
\end{figure}

A tree is \textbf{trivalent} if every vertex has degree $1$ or $3$ and at least one vertex of degree $3$, and it is \textbf{at least trivalent} or \textbf{stable} if it has no vertices of degree $2$ and at least one vertex of degree $\ge 3$.  The dual tree of any stable genus $0$ curve is a stable tree.   We define the \defn{extra valency} of a stable tree $T$ with set of internal vertices $V$ to be $\sum_{v\in V} (\deg(v)-3)$.

The \textbf{interior} of $\Mbar_{0,S}$ is the open set $M_{0,S}\subset \Mbar_{0,S}$ consisting of all the curves that have a single $\PP^1$ with all distinct marked points. The points of the interior correspond to those whose dual tree consists of a central node with $|S|$ leaves attached. 

The \textbf{boundary} of $\Mbar_{0,S}$ is the complement of the interior, consisting of the points corresponding to stable curves with more than one irreducible component. Given a set partition $S = A \sqcup B$ with $|A|, |B| \geq 2$, the \textbf{boundary divisor} $D(A|B)$ is the closure of the set of stable curves $C$ with two components, such that the marked points in $A\subset S$ are on one component and the marked points in $B\subset S$ are on another. The boundary of $\Mbar_{0,S}$ is the union of the divisors $D(A|B)$ for all choices of $A$ and $B$. Sometimes we abuse notation and write $D(A|B)$ for the associated class in the Chow ring.

Let $T$ be an at-least-trivalent tree whose leaves are labeled by $S$.  Then the \textbf{boundary stratum} $X_T$ corresponding to $T$ is the closure of the set of all stable curves whose dual tree is $T$. Let $V$ be the set of non-leaf vertices of $T$, and for each $v \in V$, let $N(v)$ be the set of vertices adjacent to $v$. The dimension of $X_T$ is the extra valency of $T$. More specifically, there is a canonical isomorphism
\begin{equation} \label{eq:stratum-factors}
X_T \cong \prod_{v \in V} \Mbar_{0,N(v)} = \prod_{v \in V} \Mbar_{0, \deg(v)},
\end{equation}
called the {\bf clutching} or {\bf gluing map}. The boundary strata $X_T$ form a quasi-affine stratification (as defined in \cite{3264}) of $\Mbar_{0,n}$, and the zero-dimensional boundary strata, or \textbf{boundary points}, correspond bijectively to the trivalent trees on leaf set $S$.  Indeed, since the points are isomorphism classes of stable curves and an automorphism of $\PP^1$ is determined by where it sends three points, a stable curve whose dual tree is trivalent represents the only element of its isomorphism class.

Keel has given a presentation of the Chow ring $A^\bullet(\Mbar_{0,n+3})$ that shows that the classes $[X_T]$ generate it as a $\mathbb{Z}$-algebra \cite{keel1992}. The relations among the $[X_T]$'s are all obtained from the basic \emph{WDVV relations} by pullback and pushforward along forgetting maps and clutching maps.

\begin{remark} \label{rmk:multiplicity}
If two sums of boundary classes $[X_T]$ are rationally equivalent, then both sums consist of the same total number of strata (counting multiplicities). This follows from Keel's presentation (and the easy fact that it holds for the WDVV relations).
\end{remark}

\subsection{Kapranov morphisms}\label{sec:Kapranov}

For all facts stated throughout the next two subsections (\ref{sec:Kapranov} and \ref{sec:background-embedding}), we refer the reader to Kapranov's paper \cite{Ka1}, in which the Kapranov morphism below was originally defined.

The \textbf{$i$th cotangent line bundle} $\mathbb{L}_i$ on $\Mbar_{0,S}$ is the line bundle whose fiber over a curve $C\in \Mbar_{0,S}$ is the cotangent space of $C$ at the marked point $i$. The $i$-th \textit{$\psi$ class} is the first Chern class of this line bundle, written $\psi_i=c_1(\mathbb{L}_i)$. The corresponding map to projective space \[|\psi_i| : \Mbar_{0,S} \to \mathbb{P}^n,\]
is called the Kapranov morphism.

We coordinatize this map as follows. It is known that $|\psi_i|$ contracts each of the $n+2$ divisors $D(\{i,j\}|\{i,j\}^c)$, for $j \ne i$, to a point $\bdpt_j \in \mathbb{P}^n$. These points are, moreover, in general linear position.  We choose coordinates so that $\bdpt_b, \bdpt_c, \bdpt_1, \ldots,\hat{\bdpt_i},\ldots, \bdpt_{n} \in \mathbb{P}^n$ are the standard coordinate points $[1:\cdots : 0], \ldots, [0:\cdots:1]$ and $\bdpt_a$ is the barycenter $[1 : 1 : \cdots : 1]$. We name the projective coordinates $[z_b : z_c : z_1 : \cdots:\hat{z_i}:\cdots : z_{n}]$.  (The notation $\hat{z_i}$ means we omit that term from the sequence.) The hyperplane $z_j = 0$ pulls back to the union of divisors $\bigcup D(i\star|aj\star^{c})$, where $\star$ ranges over the nonempty subsets of $S \setminus \{a,i,j\}$.
   
Given a curve $C$ in the interior $M_{0,S}$, by abuse of notation we also write $p_a,p_b,p_c,p_1,\dots,p_n$ for the coordinates of the $n+3$ marked points on the unique component of $C$, after choosing an isomorphism $C\cong \bP^1$. With these coordinates, the restriction of $|\psi_i|$ to the interior $M_{0,S}$ is given by
\begin{equation} \label{eq:kap-interior}
    |\psi_i|(C) = \bigg[
    \frac{p_a - p_b}{p_i - p_b} : 
    \frac{p_a - p_c}{p_i - p_c} : 
    \frac{p_a - p_1}{p_i - p_1} : 
    \cdots :
    \frac{p_a - p_{n}}{p_i - p_{n}}
    \bigg]
\end{equation}
where we omit the (undefined) term $\frac{p_a-p_i}{p_i-p_i}$.  It is convenient to choose coordinates on $C$ in which $p_a = 0$ and $p_i = \infty$, in which case the map simplifies to
\begin{equation} \label{eq:kap-interior-2}
    |\psi_i|(C) = [p_b : p_c : p_1 : \cdots :\hat{p_i}:\cdots : p_{n}].
\end{equation}

We now describe how to use the above formulas to compute $|\psi_i|$ on boundary strata, i.e. reducible stable curves $C$. Essentially, $|\psi_i|$ reduces to a smaller Kapranov morphism using the irreducible component of $C$ containing $p_i$ (followed by a linear map into $\mathbb{P}^n$).

\begin{definition}[Branches at $i$]\label{def:branches}
Let $C$ be a stable curve with dual tree $T$.  Let $v_i \in T$ be the internal vertex adjacent to leaf edge $i$.  We refer to the connected components of $T \setminus \{v_i\}$ (defined by vertex deletion) as the \defn{branches of $T$ at $i$}.  The \textbf{root} of a branch is the vertex attached to $v_i$ by an edge.  We write $\sigma(C)$ to denote the set partition of $S\setminus i$ given by the equivalence relation of being on the same branch.
\end{definition}

\begin{example}\label{ex:branches}
The stable curve $C$ below at left has the dual tree shown at center, with its disconnected branches at $i=4$ shown at right.  
\begin{center}
    \includegraphics{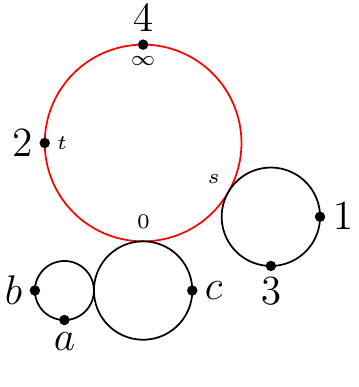}\hspace{2cm}\includegraphics{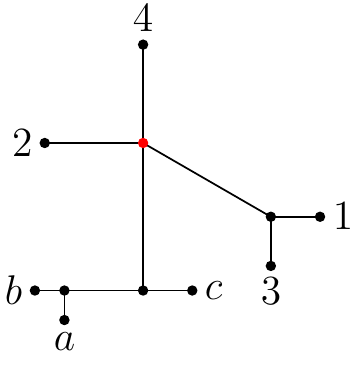}\hspace{2cm}\includegraphics{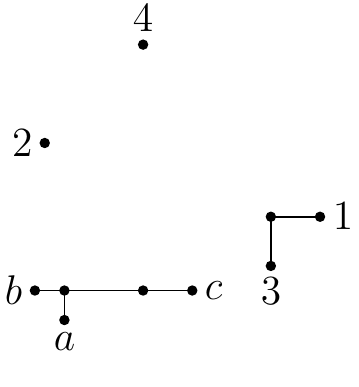}
\end{center}
By examining the branches, we find the set partition for $i=4$ is $\sigma=\{\{a,b,c\},\{2\},\{1,3\}\}$.
\end{example}

\begin{definition}\label{def:P-sigma}
Let $\partition$ be a partition of $S\setminus i$.  

Define $P^\circ_{\partition} \subset \mathbb{P}^n$ to be the set of points such that:
\begin{itemize}
    \item $z_r = z_s$ if and only if $r, s$ are in the same part of $\partition$, and 
    \item $z_r = 0$ if and only if $r,a$ are in the same part of $\sigma$.
\end{itemize} Let $P_\partition = \overline{P^\circ_\partition}$ be its closure.  It is convenient to parametrize $P_\partition$ as follows: we choose an ordering $\sigma_0,\ldots,\sigma_k$ of the parts of $\sigma$ with $a\in \sigma_0$, and for $r\in S\setminus i$ we define $\sigma(r)$ to be the index $j$ such that $r\in \sigma_{j}$. We then have the linear map
\begin{align*}
    \iota_\partition : \mathbb{P}^{k-2} &\to P_\partition \subset \mathbb{P}^n, \\
    [y_1 : \cdots : y_{k-1}] &\mapsto [y_{\sigma(b)} : y_{\sigma(c)} : y_{\sigma(1)} : \cdots:\hat{y}_{\sigma(i)}:\cdots : y_{\sigma(n)}],
\end{align*}
where $y_0$ is defined to be $0$ (that is, 
if $r\in \sigma_0$ then $z_r=0$).
\end{definition} 

\begin{example}\label{ex:minus-2}
  Let $\partition=\{\{a,b,c\},\{1,3\},\{4\}\}$, a set partition of $S\setminus 2$ for $n=4$.  Then a point of $P_\partition \subset \PP^4$ has the form $$[0:0:y_1:y_1:y_2]$$ for $y_1$ and $y_2$ not both zero.
\end{example}

\begin{prop} \label{lem:kap-arbitrary-tree}
Let $C \in \Mbar_{0,S}$ be a stable curve with dual tree $T$, and let $\sigma=\sigma(C)$ be the set partition given by the branches of $T$ at $i$. Let $C' \subseteq C$ be the irreducible component containing $p_i$, with special points $Y$. We may think of $C'$ as an interior point of the smaller moduli space $\Mbar_{0,Y}$, and compute $|\psi_i|(C')$ accordingly by \eqref{eq:kap-interior}. 
Then we have
\[|\psi_i|(C) = \iota_\partition \circ |\psi_i|(C').\]
In other words, the coordinates of \eqref{eq:kap-interior} are copied into the coordinates $\mathbb{P}^n$ according to the set partition $\partition$.
\end{prop}

\begin{example}
Let $C$ be the curve in Example \ref{ex:branches}, and let $C'$ be the component containing marked point $4$.  If we parameterize $C'\cong \PP^1$ such that branch $\{4\}$ is at $\infty$, branch $\{a,b,c\}$ is at $0$, and $\{2\}$ and $\{1,3\}$ are at $t$ and $s$ respectively, then
$$|\psi_4|(C)=[0:0:s:t:s].$$ 
\end{example}

\subsection{The total and iterated Kapranov maps}\label{sec:background-embedding}

We can now define the maps $\Psi_n$ and $\Omega_n$.

\begin{definition}
We define $\Psi_n:\Mbar_{0,S}\to \PP^n\times \PP^n\times \cdots \times \PP^n$ to be the product $|\psi_1|\times |\psi_2|\times \cdots \times |\psi_n|$.  That is, $$\Psi_n(C)=(|\psi_1|(C),|\psi_2|(C),\ldots,|\psi_n|(C)).$$
\end{definition}

The map $\Psi_n$ is not an embedding, since it only records the coordinates of special points on components $C' \subseteq C$ containing at least one marked point $i \geq 1$. However, $\Psi_n$ is birational onto its image (indeed even a single $|\kap_i|$ map is birational onto its image).

\begin{example}\label{ex:Psi}
 If $C$ is the curve in Example \ref{ex:branches}, we have $$\Psi_n(C)=([0:0:0:1:0],[0:0:s:s:s-t],[0:0:1:0:0],[0:0:s:t:s])$$ where the second coordinate $|\psi_2|(C)$ is obtained by combining Lemma \ref{lem:kap-arbitrary-tree} and Equation \eqref{eq:kap-interior}, using the same parameterization of the red component $C'$ for both $|\psi_2|$ and $|\psi_4|$.  Note that the coordinates in the second copy of $\PP^4$ match the format shown in Example \ref{ex:minus-2}.
\end{example}

To define $\Omega_n$, we can combine the $\psi$ and forgetting maps as follows.  The Kapranov morphism is a projective embedding of the universal curve over $\Mbar_{0,S\setminus n}$:
\[
\xymatrix{
\Mbar_{0,S}\ \ar@{^(->}[r]^-{\kapn} \ar[d]_{\pi_n} & 
\ \mathbb{P}^n \times \Mbar_{0, S\setminus n} \ar[dl] \\
\Mbar_{0,S\setminus n}.
}
\]
We may repeat this construction using the map $|\psi_{n-1}|$ on $\Mbar_{0, S\setminus n}$, and so on, obtaining a sequence of embeddings. This gives the {\bf iterated Kapranov morphism}
\[\emb_n: \Mbar_{0,S} \hookrightarrow \mathbb{P}^1 \times \mathbb{P}^2 \times \cdots \times \mathbb{P}^n.\]
Keel and Tevelev \cite{KeT} first observed that $\Omega_n$ is in fact a closed embedding.
The $i$-th factor of this embedding is given by forgetting the points $p_{i+1}, \ldots, p_n$, then applying the Kapranov morphism $|\psi_i|$ on the smaller moduli space.  Since the $\omega$ classes are defined as the pullbacks of $\psi$ classes under the forgetting maps, we may alternatively define $$\Omega_n=|\omega_1|\times \cdots \times |\omega_n|.$$

\begin{example}
  If $C$ is the curve in Example \ref{ex:branches}, we have $$\Omega_n(C)=([0:1],[0:0:1],[0:0:1:0],[0:0:s:t:s]).$$
\end{example}

\begin{remark}
Example \ref{ex:Psi} demonstrates that $\Psi_n$ is not an embedding.  Indeed, if we replace the $\{a,b,c\}$ branch of the curve with any other arrangement of $a,b,c$ with respect to each other, the resulting curve will have the same coordinates under $\Psi_n$.   On the other hand, since $\Omega_n$'s coordinates are computed after applying forgetting maps at each step, there will exist a step where a numbered marked point will ``see'' the structure of such an ambiguous branch.  Hence $\Omega_n$ is injective.
\end{remark}

\section{Slide rules}\label{sec:slide-rules}

In this section, we define the slide rules for $\psi$ and $\omega$. We first state each rule as a generative procedure for generating a list of trees. We also describe the resulting sets of trees directly in terms of edge labelings. We prove in Section~\ref{sec:limiting-hyperplanes} that the trees (strata) given by these rules compute the products $\omega^{\mathbf{k}}$ and $\psi^\mathbf{k}$.

Let $T$ be a stable (at-least trivalent) tree with leaves labeled $a < b < c < 1 < \cdots < n$.

\begin{definition}
Fix $1 \leq i \leq n$ and let $v_i \in T$ be the internal vertex adjacent to $i$. We write $\mathrm{Br}_a$ be the branch at $i$ containing $a$. We write $e_a$ for the edge connecting $\mathrm{Br}_a$ to $v_i$. 
\end{definition}

\begin{definition}
With $i$ as above, let $m$ be the minimal leaf label of $T\setminus (\mathrm{Br}_a\cup \{i\})$; we call $m$ the \defn{$i$-minimal marked point}.  We write $\mathrm{Br}_m$ to denote the branch at $i$ containing $m$.
\end{definition}

\begin{definition}[Slide at $i$]\label{def:slide-rule}
An $i$-{\bf slide} on $T$ is performed as follows: with the notation above, we add a vertex $\overline{v}$ in the middle of edge $e_a$,  move $\mathrm{Br}_m$ to attach its root to $\overline{v}$, and attach each remaining branch of $T$ at $i$ (other than $\mathrm{Br}_a$) to either $v_i$ or $\overline{v}$. 

We write $\slide_i(T)$ for the set of stable trees obtained this way. Note that stability requires at least one branch to remain at $v_i$. In particular, $\slide_i(T)$ is empty if $\deg(v_i) = 3$.
\end{definition}

\begin{remark}\label{rmk:slide-contraction}
It is straightforward to check that $\slide_i(T)$ can alternatively be defined as the set of all trees $T'$ for which:
\begin{itemize}
    \item Contracting a single edge $e$ in $T'$ results in $T$ (in the above notation, the edge connecting $\overline{v}$ and $v_i$), and
    \item The leaves $a$ and $m$ are on the same branch with respect to $i$ in $T'$.
\end{itemize}
\end{remark}

\begin{example}
 As an example of a $3$-slide, let $T$ be the following tree, along with the new vertex $v_m$ to be added to edge $e_a$ as shown below.  We also indicate the vertex $v_i=v_3$ with a dot.
 \[\includegraphics{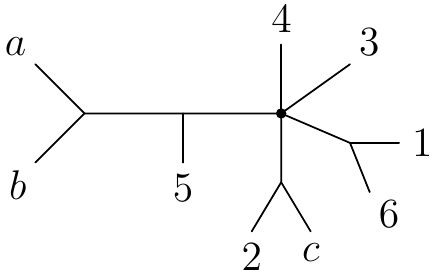}\]
 Then $\mathrm{Br}_a$ is the subtree having leaves $a,b,5$.  The other branches at $3$ have sets of leaves $\{4\}$, $\{1,6\}$, and $\{c,2\}$, and since the latter has the smallest minimal element ($m=c$) among these branches, $\mathrm{Br}_m$ is the branch containing $c$ and $2$. Performing the $3$-slide gives us the set of three trees:
\begin{center}
    \includegraphics{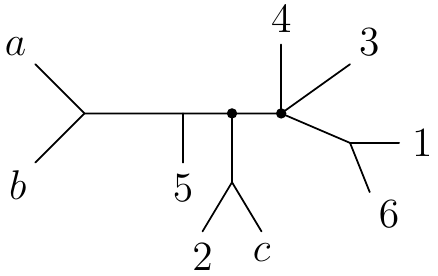}\hspace{0.5cm}\includegraphics{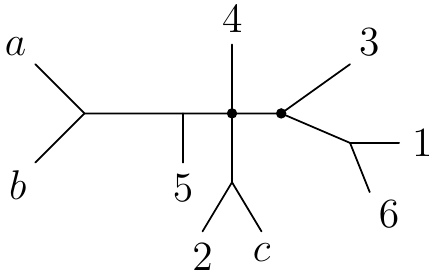}\hspace{0.5cm}\includegraphics{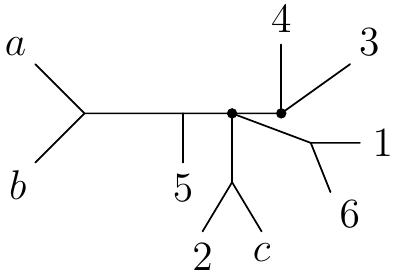}
\end{center}
\end{example}

\begin{remark}
In general, there are $2^{\deg(v_i)-3}-1$ elements in $\slide_i(T)$.  Indeed, each branch other than:
\begin{itemize}
    \item branch $\mathrm{Br}_a$,
    \item the leaf $i$, and
    \item branch $\mathrm{Br}_m$,
\end{itemize}  
has the choice of either being attached to $v_i$ or $\overline{v}$, with the exception that they cannot all be attached to $\overline{v}$.
\end{remark}

The following lemma about $i$-slides, while straightforward, is essential to the generic reducedness result.

\begin{lemma}[Injectivity] \label{lem:distinct-insertions}
Let $T, T'$ be distinct stable trees on leaf set $S$.  Then the sets $\slide_{i}(T)$ and $\slide_{i}(T')$ are disjoint.
\end{lemma}

\begin{proof}
Let $R \in \slide_i(T)$. Let $v_i \in R$ be the vertex where $i$ is attached. Let $e_A \in R$ be the edge adjacent to $v_i$ connecting to the branch from $v_i$ containing $a$. Contracting $e_A$ recovers $T$.
\end{proof}

We now define the general slide rules for intersections of $\psi$ and $\omega$ classes.  In both of the following we let $\mathbf{k} = (k_1, \ldots, k_n)$ be a weak composition.  We write $\abctree$ (resp. $\abcntree$) for the unique tree with a single internal vertex and leaves $a,b,c$ (resp. $a,b,c,1,\ldots,n$).

\begin{definition}[Slide rules for $\psi$]\label{def:slide-psi}
 We define $\Slide^\psi(\mathbf{k})$ as the set of all stable trees obtained as follows.
\begin{itemize}
    \item[1.] Start with $\abcntree$ as step $i=0$.
    \item[2.] For $i = 1, \ldots, n$, perform $k_i$ successive $i$-slides in all possible ways starting from the trees obtained in step $i-1$.
\end{itemize}
\end{definition}
\begin{definition}[Slide rules for $\omega$]\label{def:slide-omega}
Define $\Slide^\omega(\mathbf{k})$ as the set of all stable trees obtained as follows.
\begin{itemize}
    \item[1.] Start with $\abctree$ as step $i=0$.
    \item[2.] For $i=1, \ldots, n$:
    \begin{itemize}
        \item[a.] Consider all trees formed by inserting $i$ at any existing non-leaf vertex on a tree obtained in step $i-1$.
        \item[b.] Perform $k_i$ successive $i$-slides in all possible ways starting from the trees obtained in the previous step.
    \end{itemize}
\end{itemize}
\end{definition}

More formally, if $\mathbf{T}$ is a set of $S$-labeled stable trees, we write
\[\slide_{i}(\mathbf{T}) := \bigcup_{T \in \mathbf{T}} \slide_{i}(T).\]
By Lemma \ref{lem:distinct-insertions}, this is a disjoint union. For $k \geq 0$, we write $\slide_{i}^{(k)}(\mathbf{T}) := \slide_{i} \circ \cdots \circ \slide_{i}(\mathbf{T})$
for the result of applying $k$ successive slides to the elements of $\mathbf{T}$ (in all possible ways).  We also write $\pi_{n+1}^{-1}(T)$ for the set of all trees $T'$ obtained by inserting $n+1$ at an internal node of $T$. (This corresponds to the geometric computation of $\pi_{n+1}^{-1}(X_T)$.)  If $\mathbf{T}$ is a set of trees, we write $\pi^{-1}_{n+1}(\mathbf{T})$ for the corresponding (evidently disjoint) union.

With this notation, we may state Definitions \ref{def:slide-psi} and \ref{def:slide-omega} formally as:
\begin{align*}
\Slide^\omega(\mathbf{k}) &:= \slide_{n}^{(k_n)} \circ \pi_n^{-1} \circ \cdots \circ \slide_{i}^{(k_i)} \circ \pi_i^{-1} \circ \cdots \circ \slide_{1}^{(k_1)} \circ \pi_1^{-1}(\abctree), \\
\Slide^\psi(\mathbf{k}) &:= \slide_{n}^{(k_n)} \circ \cdots \circ \slide_{i}^{(k_i)} \circ \cdots \circ \slide_{1}^{(k_1)} (\abcntree).
\end{align*}

We illustrate the slide rule for $\mathbf{k}=(1,0,2)$ for both $\psi$ and $\omega$ in the next two examples.

\begin{example}\label{ex:slide-psi}
As an example, we compute $\Slide^\psi(1,0,2)$. We first start with $\sixtree$, the unique tree with a single internal vertex and six leaves labeled $a,b,c,1,2,3$. We then perform one $1$-slide to obtain the trees:
\begin{center}
    \includegraphics{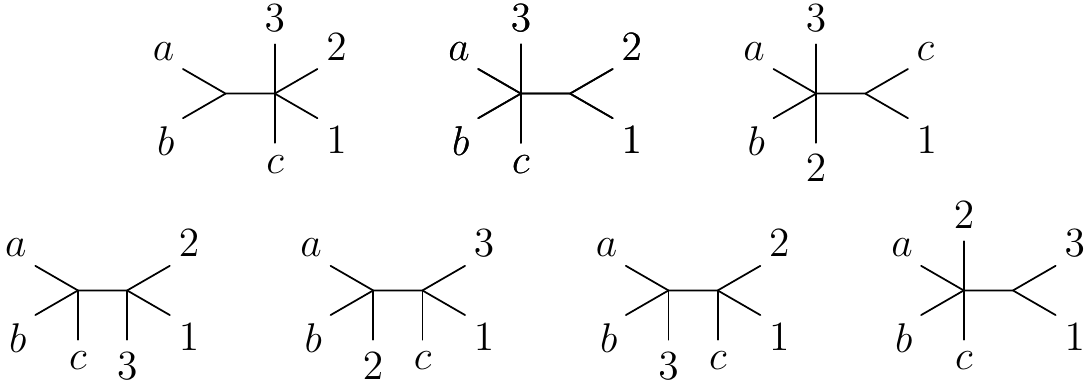}
\end{center}
and then apply two $3$-slides to each of these.  Notice that we can only perform a $3$-slide when the vertex that leaf $3$ is attached to has degree greater than three.  In particular, only the trees in the top row shown above will generate nonempty sets after two $3$-slides.  Performing two $3$-slides on these trees yields the three trivalent trees:
\begin{center}
    \includegraphics{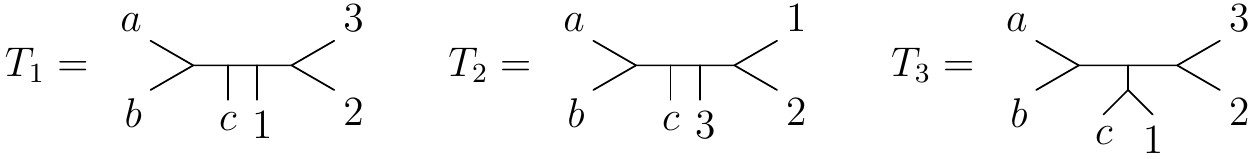}
\end{center}
Thus $\Slide^{\psi}(1,0,2)=\{T_1,T_2,T_3\}$.
\end{example}

\begin{remark} \label{rmk:slides-vertex-degree-bound}
Notice that, at any given step in the slide algorithm, a tree $T$ can be ignored if, for any vertex $v \in T$, the total number of remaining slides for all leaves $i$ adjacent to $v$ is greater than $\deg(v) - 3$. The slides starting from such a tree will eventually result in the empty set. This can also be seen geometrically for dimension reasons, using the factorization in Equation \eqref{eq:stratum-factors}.
\end{remark}

\begin{example}\label{ex:slide-omega}
For comparison, we now compute $\Slide^{\omega}(1,0,2)$.  We start with $\abctree$ and at step $1$ insert the $1$ at an internal vertex in all possible ways (which is only one possible way in this case).  We then perform a $1$-slide:
\begin{center}
    \includegraphics[scale=0.8]{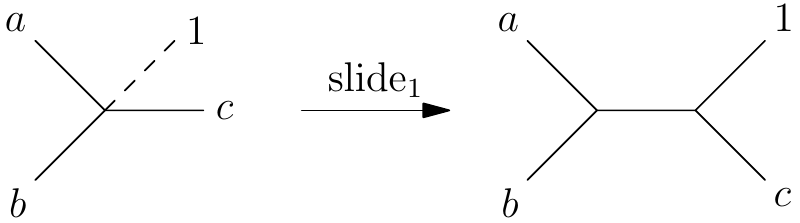}
\end{center}
We then insert $2$ in all possible ways (and do not performing any $2$-slides), then insert $3$ in all possible ways afterwards. We reach the four trees below: 
\begin{center}
    \includegraphics[scale=0.8]{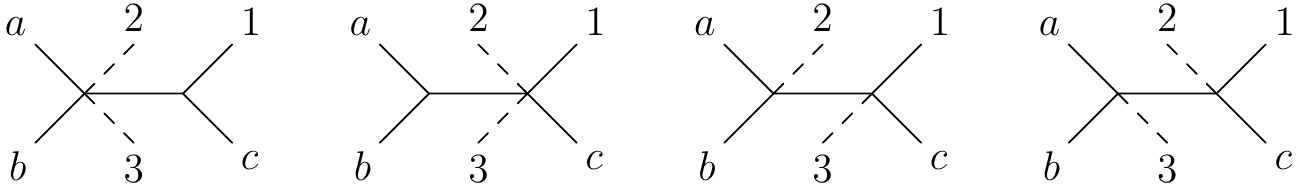}
\end{center}
We finally perform two $3$-slides starting from each of these trees; the two on the right produce the empty set, and the two on the left map to trees $T_3$ and $T_1$ from Example \ref{ex:slide-psi}.  Thus $\Slide^{\omega}(1,0,2)=\{T_1,T_3\}$.
\end{example}

In addition to the generative procedure above, it is also convenient to have a criterion to say directly when a given stable tree $T$ is in $\Slide^\psi(\mathbf{k})$ or $\Slide^\omega(\mathbf{k})$.

\begin{definition}\label{def:slide-labeling}
  The ($\omega$ or $\psi$) \textbf{$\mathbf{k}$-slide labeling} of $T$, if it exists, is formed by the following process (and if the process terminates before completion, it does not exist).  Set $\ell=n$.
  \begin{enumerate}
  \item \textbf{Contract labeled edges.} Let $T'$ be the tree formed by contracting all internal edges of $T$ that are already labeled.
  \item \textbf{Identify the next edge to label.} In $T'$, let $v_\ell$ be the internal vertex adjacent to leaf edge $\ell$. Let $e$ be the first edge on the path from $v_\ell$ to $a$, and let $\overline{v}$ be the other vertex of $e$.  If $\overline{v}=a$, the process terminates; otherwise go to the next step.
  \item \textbf{If minimal values decrease, label the edge.} Define $m_{v_\ell}$ (resp.\ $m_{\overline{v}}$) to be the smallest label on any branch from $v_\ell$ (resp.\ $\overline{v}$) not containing $a$ or $\ell$.  If $\ell>m_{v_\ell}>m_{\overline{v}}$ in the $\omega$ case, or if simply $m_{v_\ell}>m_{\overline{v}}$ in the $\psi$ case, then label edge $e$ by $\ell$ (in both $T'$ and $T$).  Otherwise, the process terminates.
  \item \textbf{Iterate.} If there are less than $k_\ell$ internal edges of $T$ labeled by $\ell$, repeat steps 1--4.  Otherwise, decrement $\ell$ by $1$. If $\ell=0$ the labeling is complete, and if $\ell>0$ repeat steps 1--4.
  \end{enumerate}
\end{definition}

\begin{thm}\label{thm:labelings}
  The sets $\Slide^\omega(\mathbf{k})$ and $\Slide^{\psi}(\mathbf{k})$ are, respectively, the sets of all trivalent trees that admit an $\omega$ or $\psi$ type $\mathbf{k}$-slide labeling.
\end{thm}

By Remark \ref{rmk:slide-contraction}, it is clear that the contraction and labeling steps simply reverse the slides in each case, and we omit the proof. 

The slide labeling interpretation allows us to easily show the following.

\begin{prop} \label{prop:omega-subset-psi}
For all compositions $\mathbf{k}$, $\Slide^\omega(\mathbf{k}) \subseteq \Slide^\psi(\mathbf{k})$.
\end{prop}

\begin{proof}
Any $\omega$-type slide labeling is also a $\psi$-type slide labeling since the inequality $l>m_{v_\ell}>m_{\overline{v}}$ is a stricter condition than simply $m_{v_\ell}>m_{\overline{v}}$ in step 3 of the slide labeling process.
\end{proof}
This containment can also be seen by `simulating' the generative procedure for $\Slide^\omega$ starting from $\abcntree$ rather than $\abctree$, excluding the leaves $j > i$ when determining the $i$-minimal marked point $m$, and requiring at least one branch containing a leaf $j' < i$ (rather than an arbitrary branch) to remain attached to $v_i$. This expresses $\Slide^\omega$ as a subset of the choices for $\Slide^\psi$.

\begin{example}
The points of $\Slide^\psi(1,0,2)$ are shown in Figure \ref{fig:slide-labeling}, along with their slide labelings.  Note that the middle tree does not admit an $\omega$-type slide labeling, because after contracting the edges labeled $3$, the $1$ compares minima $2$ vs $c$, and while $2>c$, it is not the case that $1>2>c$.  Therefore it only admits a $\psi$-type labeling for $(1,0,2)$ and not an $\omega$-type labeling.
\end{example}

\begin{figure}
    \centering
    \includegraphics{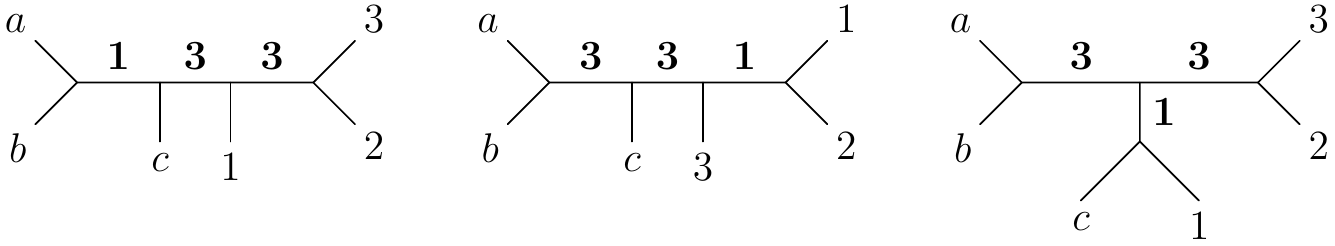}
    \caption{The three points of $\Slide^\psi(1,0,2)$, along with their $\psi$-type slide labelings.  In the third labeling above, we think of the edges labeled $3$ as contracted before trying to label the third edge by $1$.  The $1$ then compares the minima of $c$ vs $b$ in the contracted tree, and hence can ``slide'' along its path towards $a$.}
    \label{fig:slide-labeling}
\end{figure}

\subsection{Nonempty slide sets}

Using the slide labeling rule, we can identify a particular tree that is in all of the (nonempty) sets $\Slide^{\omega}(k_1,\ldots,k_n)$ for $k_1+\cdots+k_n=n$, and in many of the sets $\Slide^{\psi}(k_1,\ldots,k_n)$. We require the following conditions to state these results.

\begin{definition}
Let $\mathbf{k}$ be a composition of $n$. We say $\mathbf{k}$ is \defn{Catalan} if, for all $i$, 
\[k_n+k_{n-1}+\cdots + k_{n-i+1}\ge i.\]
We say $\mathbf{k}$ is \defn{almost-Catalan} for all $i$, 
\[k_n+k_{n-1}+\cdots + k_{n-i+1}\ge i-1.\]
\end{definition}

\begin{prop}\label{prop:common-tree}
Let $T_0$ be the tree
\begin{center}
    $T_0 =$ \raisebox{-.5\height}{\includegraphics{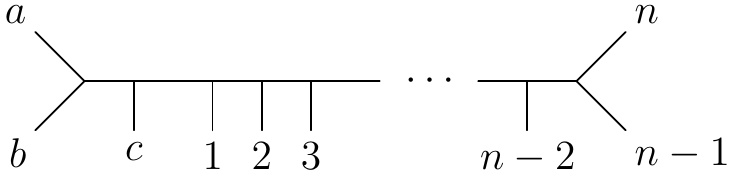}}.
\end{center}
Then $T_0 \in \Slide^\omega(\mathbf{k})$ if and only if $\mathbf{k}$ is Catalan, and $T_0 \in \Slide^\psi(\mathbf{k})$ if and only if $\mathbf{k}$ is almost-Catalan.
\end{prop}
\begin{proof}
Let $e_c, e_1, e_2,\ldots,e_{n-1}$ be the internal edges in $T_0$ above from left to right.

For $\omega$, the slide labeling is valid if and only if, just before an edge is labeled by $i$, the $i$-minimal element (after contracting previously labeled edges) is less than $i$.  This occurs if and only if some larger label $j>i$ labels the edge $e_i$ before we begin labeling edges by $i$.  In addition, all edges to the right of $e_i$ must have labels larger than $i$ as well, since the edge labelings occur along the paths towards $a$.  Thus the total number of edges labeled before step $i$, which is given by $k_n+k_{n-1}+\cdots+k_{i+1}$, is at least as large as the number of internal edges to the right of vertex $v_{i-1}$, namely, $n-i$.  Thus we have $$k_{n}+k_{n-1}+\cdots+k_{i+1}\ge n-i$$ for all $i$.  Since $k_1+\cdots+k_n=n$, this is equivalent to the Catalan condition.

For $\psi$, the same argument as above holds except that $e_i$ does not have to be labeled by something larger than $i$, and so we only need $k_{n}+k_{n-1}+\cdots+k_{i+1}\ge n-i-1$, which is equivalent to the almost-Catalan condition.
\end{proof}

\begin{prop}\label{prop:Catalan}
For a composition $\mathbf{k}$ with $k_1+k_2+\cdots+k_n=n$, the set $\Slide^{\omega}(\mathbf{k})$ is nonempty if and only if $\mathbf{k}$ is Catalan.
\end{prop}

While this follows from Corollary \ref{cor:multidegrees} combined with the combinatorial results on multidegrees in \cite{CGM}, we give a direct combinatorial proof here.

\begin{proof}
Note that the extra valency (see Section \ref{sec:background-trees}) of all trees at a given step of the slide rule algorithm is a fixed constant; indeed, inserting a new leaf increases the extra valency by $1$, and applying $\slide_i$ decreases it by $1$.  In particular, after step $i$ we have a set of trees having extra valency $i-(k_1+k_2+\cdots+k_i)$.

Now, suppose $\Slide^{\omega}(\mathbf{k})$ is nonempty. Then since the extra valency at step $i$ is $i-(k_1+k_2+\cdots+k_i)$, we have $i\ge k_1+k_2+\cdots+k_i$ for all $i$, and a simple algebraic manipulation (along with the fact that $k_1+k_2+\cdots+k_n = n$) shows that this is equivalent to the Catalan condition.

The converse follows from Proposition \ref{prop:common-tree}.
\end{proof}

\begin{remark}
The sets $\Slide^{\psi}(\mathbf{k})$ are nonempty for all $\mathbf{k}$ with $\sum k_i\le n$, since the extra valency at each step is $n-(k_1+\cdots+k_i)$, and the valency can always be distributed in each slide to guarantee that before the $i$th slide the vertex attached to $i$ has degree at least $k_{i}+3$.
\end{remark}

\section{Limiting hyperplanes on \texorpdfstring{$\Mbar_{0,n}$}{} and \texorpdfstring{$\psi$}{} and \texorpdfstring{$\omega$}{} product formulas}\label{sec:limiting-hyperplanes}

We now show that the trees in $\Slide^{\omega}(\mathbf{k})$ and $\Slide^{\psi}(\mathbf{k})$ describe boundary strata representing, respectively, the cycle classes $\omega^\mathbf{k} := \omega_1^{k_1} \cdots \omega_n^{k_n}$ and $\psi^\mathbf{k} := \psi_1^{k_1} \cdots \psi_n^{k_n}$.

We will do this by constructing an explicit flat limit of hyperplanes. We start with necessary general preliminaries on flat limits.

\subsection{Flat limits}
\newcommand{\specDVR}{\mathrm{Spec}(R)}

Let $M$ be a smooth projective variety, $T$ a smooth curve (we will always use $\mathbb{A}^1$ or an open subset thereof), $0 \in T$ a closed point, and $t\in T$ the generic point. Let $V \subseteq M \times T$ be a closed subscheme. We write $V_0$ for the fiber over $0$ and $V_t$ for the generic fiber.

The \emph{flat limit of $V_t$ as $t \to 0$} is by definition the fiber of the scheme-theoretic closure,
\[ \lim_{t \to 0} V_t := (\overline{V|_{T - \{0\} }}) |_0.\]
Algebraically, the limit is given by saturating the ideal of $V$ with respect to $t$, then setting $t=0$. In general we have \[\lim_{t \to 0} V_t \subseteq V_0,\]
but equality need not hold; in fact it holds (scheme-theoretically) if and only if $V$ is flat over a neighborhood of $0 \in T$. See \cite[Proposition III.9.8]{Hartshorne}.

Below, our approach will involve calculating the cycle class of a flat limit by finding an ``almost-transverse'' $V_0$ that equals it generically.  A scheme $X$ is \emph{generically reduced} if it is reduced on some dense open subscheme; in this case, all the irreducible components of $X$ have multiplicity 1.  We also say $X$ has \emph{pure dimension $d$} if all of its irreducible components have the same dimension $d$.  

We recall the following fact about transversality and intersection products:

\begin{prop} \label{prop:gen-reduced-intersection}
Let $M$ be a smooth variety (not necessarily proper) and $X, X' \subseteq M$ subschemes of pure codimensions $c, c'$. Suppose $X \cap X'$ is of pure codimension $c+c'$ and is generically reduced. Then $[X \cap X'] = [X] \cdot [X']$.
\end{prop}

\begin{proof}
By \cite[Prop 8.2(a)]{fulton:it}, each irreducible component $Z \subseteq X \cap X'$ occurs in $[X] \cdot [X']$ with coefficient between $1$ and the scheme-theoretic multiplicity of $Z$ in $X \cap X'$. Generic reducedness says that this multiplicity is also $1$.
\end{proof}

The next lemma is a ``generically reduced'' version of Lemma 37.24.6 in the Stacks project \cite[\href{https://stacks.math.columbia.edu/tag/0574}{Tag 0574}]{stacks-project}, which is the analogous result for \textit{reduced} fibers.

\begin{lemma}\label{lem:gen-reduced-fiber}
Let $V \to T$ be flat and proper over a neighborhood of $0 \in T$. Assume $V$ is pure of dimension $d$. If $V_0$ is generically reduced, so is $V_t$.
\end{lemma}

\begin{proof}
Let $Z \subseteq V_t$ be an irreducible component and let $\overline{Z}$ be its closure in $V$.  Since $t$ is the generic point of $T$, $\overline{Z} \to T$ is dominant and flat; by properness the image contains $0 \in T$, so $\overline{Z} \cap V_0$ is nonempty. Hence by flatness $\overline{Z} \cap V_0$ is of pure dimension $d-1$.

Let $Z' \subseteq \overline{Z} \cap V_0$ be an irreducible component. By assumption, $V_0$ is reduced and \emph{smooth} along some dense open subset $U\subseteq V_0$.  Let $x \in U\cap Z'$ be a closed point (which must exist since $U$ is dense and $Z'$ is an irreducible component). Then the Zariski tangent space to $V_0$ at $x$ has dimension exactly $d-1$. Since $V_0$ is locally cut out in $V$ by the single equation $t=0$, the Zariski tangent space to $V$ at $x$ has dimension $\leq (d-1) + 1 = d$. Since this matches the Krull dimension of $V$, it follows that $x$ is a smooth, in particular reduced, point of $V$. Therefore $\overline{Z}$ is actually smooth and reduced at $x$, hence is generically (smooth and) reduced. Since $Z$ was arbitrary, it follows that $V_t$ is generically reduced.
\end{proof}

We will need the following statement about ``almost-transversality'' for dynamic intersections, a criterion for the flat limit to be generically reduced.

\begin{lemma} \label{lem:generically-equal-limits}
Let $M$ be a smooth projective variety, $T$ a smooth curve and $0 \in T$. Let $V \subseteq M \times T$ be a subscheme, flat over $T$ and pure of relative dimension $d$. Let $\psi : M \to \PP^n$ be a map and $H \subseteq \mathbb{P}^n$ a hypersurface.

Suppose $\psi^{-1}(H) \cap V_0$ is generically reduced and of pure dimension $d-1$. Then $\displaystyle{\lim_{t \to 0}(\psi^{-1}(H) \cap V_t)}$ is generically reduced and has the same underlying set as $\psi^{-1}(H) \cap V_0$.
\end{lemma}

\begin{proof}
Write $\displaystyle{F_0 = \lim_{t \to 0} (\psi^{-1}(H) \cap V_t)}$ for the flat limit. We first check that $F_0$ is pure of dimension $d-1$. By flatness, it is enough to show that $\psi^{-1}(H) \cap V_t$ is pure of dimension $d-1$. Fiber dimension is upper semi-continuous for proper maps (\cite[Theorem 11.4.2]{Vakil}), so 
\[\dim (\psi^{-1}(H) \cap V_t) \leq \dim(\psi^{-1}(H) \cap V_0) = d-1.\]

Conversely, since $\psi^{-1}(H)$ is a Cartier divisor, $\psi^{-1}(H) \cap V_t$ is given by a principal ideal on $V_t$, so by Krull's principal ideal theorem and the purity of $V_t$, every component of $\psi^{-1}(H) \cap V_t$ has dimension $\geq \dim(V_t) - 1 = d-1$. Thus, $\psi^{-1}(H) \cap V_t$ is pure of dimension $d-1$ as required.

Next, since $F_0 \subseteq \psi^{-1}(H) \cap V_0$ and $\psi^{-1}(H) \cap V_0$ is generically reduced and both are of the same (pure) dimension, $F_0$ is also generically reduced.

Finally, we show that $F_0$ agrees set-theoretically with $\psi^{-1}(H) \cap V_0$, i.e. $\psi^{-1}(H) \cap V_0$ does not have extra components compared to $F_0$. It suffices to show that the fundamental cycles $[F_0]$ and $[\psi^{-1}(H) \cap V_0]$ are the same. We have
\begin{equation} \label{eq:zero-eq-prod}
[\psi^{-1}(H) \cap V_0] = [\psi^{-1}(H)] \cdot [V_0] \end{equation}
by Proposition \ref{prop:gen-reduced-intersection} and our assumption on $\psi^{-1}(H) \cap V_0$. Also, by Lemma \ref{lem:gen-reduced-fiber}, since $F_0$ is generically reduced, so is $\psi^{-1}(H) \cap V_t$, so by Proposition \ref{prop:gen-reduced-intersection} a second time, 
\begin{equation} \label{eq:t-eq-prod}
[\psi^{-1}(H) \cap V_t] = [\psi^{-1}(H)] \cdot [V_t].
\end{equation}
Lastly, by \cite[Corollary 11.1]{fulton:it}, the limit intersection class satisfies
\begin{equation} \label{eq:lim-eq-prod}
\lim_{t\to 0} \Big( [\psi^{-1}(H)] \cdot [V_t] \Big) =
[\psi^{-1}(H)] \cdot [V_0].
\end{equation}
Combining, we have
\begin{align}
[F_0] := \lim_{t \to 0}\ [\psi^{-1}(H) \cap V_t] 
&= \lim_{t\to 0} \Big( [\psi^{-1}(H)] \cdot  [V_t] \Big) \
\text{by } \eqref{eq:t-eq-prod},\\
&= [\psi^{-1}(H)] \cdot [V_0] \ 
\text{by } \eqref{eq:lim-eq-prod}, \\
&= [\psi^{-1}(H) \cap V_0] \
\text{by } \eqref{eq:zero-eq-prod}.
\end{align}
This completes the proof.
\end{proof}

We note that these hypotheses do \emph{not} imply $\psi^{-1}(H) \cap V_0 = F_0$ scheme-theoretically, as the following example illustrates.

\begin{example}
Let $\mb{P}^3$ have coordinates $[x:y:z:w]$, and let $V \subset \mb{P}^3 \times \mr{Spec} k[t]$ be defined by the ideal \[(x) \cap (x, y-tw, z-tw)^2,\] that is, $V_t$ is the plane $x=0$ with an embedded nonreduced point located at $p = [0:t:t:1]$. Let $H$ be the hyperplane $y=0$. Then $H \cap V_0$ is the line $x=y=0$ with an embedded point at $[0:0:0:1]$, whereas the flat limit $F_0 = \lim_{t \to 0}(H \cap V_t)$ is the reduced line $x=y=0$. However, $F_0$ and $H \cap V_0$ are generically equal.
\end{example}

We will apply Lemma \ref{lem:generically-equal-limits} repeatedly to analyze iterated limits, in the following form.

\begin{lemma} \label{lem:generically-equal-limits-iterated}
Let $V \subseteq M \times T$ be a closed subscheme, flat over $T$ and pure of relative dimension $d$. Let $\psi : M \to \PP^n$ be a map and let $H \subset \mb{P}^n \times T$ be a flat family of hypersurfaces. Suppose $\psi^{-1}(H_0) \cap V_0$ is generically reduced and of pure dimension $d-1$.  

Then $\displaystyle{\lim_{s \to 0} \lim_{t \to 0} \big( \psi^{-1}(H_s) \cap V_t}\big)$ is generically reduced and, \emph{set-theoretically}, we have the equality
\[
\lim_{s \to 0} \lim_{t \to 0} \big( \psi^{-1}(H_s) \cap V_t \big) 
= \lim_{s \to 0} \big(\psi^{-1}(H_s) \cap \lim_{t \to 0} V_t\big) 
\ \ \bigg(= \lim_{s \to 0} (\psi^{-1}(H_s) \cap V_0) \bigg).\]
That is, we may ``pull the $H_s$ past the $\displaystyle{\lim_{t \to 0}}$'' without changing the generic scheme structure.
\end{lemma}

\begin{proof}
Since $\psi^{-1}(H_0) \cap V_0$ is generically reduced and of the correct dimension, the same is true for $\psi^{-1}(H_s) \cap V_0$ by semicontinuity (as in the proof of Lemma \ref{lem:generically-equal-limits}). Applying Lemma \ref{lem:generically-equal-limits}, we see $\displaystyle{\lim_{t \to 0} \psi^{-1}(H_s) \cap V_t}$ is generically reduced and has the same underlying set as $\psi^{-1}(H_s) \cap V_0$. Therefore the limits of each as $s \to 0$ are again generically equal.
\end{proof}

Finally, flat limits are preserved by flat pullbacks:

\begin{lemma}\label{lem:flat-pullbacks}
Let $f : V \to W$ be a flat morphism of projective varieties. Let $X \subset W \times \specDVR$ be a subscheme. Then
\[f^{-1}\big(\lim_{t \to 0} X_t \big) = \lim_{t \to 0} \big( f^{-1}(X_t) \big).\]
\end{lemma}
\begin{proof}
We have $f^{-1}(X|_{\specDVR \setminus 0}) = f^{-1}(X)|_{\specDVR \setminus 0}$. Flat pullback preserves closures, so
\[f^{-1}(\overline{X|_{\specDVR \setminus 0}}) = \overline{f^{-1}(X)|_{\specDVR \setminus 0}}.
\]
Setting $t=0$ gives
\[f^{-1}\big(\lim_{t \to 0} X_t \big) = \lim_{t \to 0} \big( f^{-1}(X)_t \big) = \lim_{t \to 0} \big( f^{-1}(X_t) \big). \qedhere\]
\end{proof}

\subsection{Limits of intersections}\label{sec:hyperplanes}

Let the $i$th factor of $\mathbb{P}^n$ in the product $\mathbb{P}^n\times \cdots \times \mathbb{P}^n$ have coordinates $[z_b :z_c:z_1: \cdots : \widehat{z_i} : \cdots : z_n]$, and let $\PP^i$ have coordinates $[w_b:w_c:w_1:\cdots:w_{i-1}]$ as in Section \ref{sec:background}. Recall from the introduction that we define
\begin{align}
\hyp_{i}^\psi(t) &= \mathbb{V}(z_b + t z_c + t^2 z_1 + \cdots + t^{i}z_{i-1} + t^{i+1}z_{i+1} + \cdots + t^n z_n), \label{eq:H-psi}\\
\hyp_i^\omega(t) &= \mathbb{V}(w_b + t w_c + t^2 w_1 + \cdots + t^{i}w_{i-1}),
\end{align}

We first examine the limit of a single hyperplane section of a stratum.  Let $\psi_i$ be the $i$-th Kapranov map $\Mbar_{0,S}\to \PP^n$.

\begin{StrataPsi}
  Let $T$ be a stable tree.
  Let $V_i(t)=|\psi_i|^{-1}(\hyp_i^\psi(t))$ in $\Mbar_{0,n+3}$. Then the limiting fiber is given by
  $$\lim_{t\to 0}(V_i(t)\cap X_T)=\bigcup_{T'\in\slide_i(T)} X_{T'},$$
  and it is reduced.
\end{StrataPsi}

\begin{proof}
Let $v_i \in T$ be the node to which $i$ is attached. Let $\partition$ be the set partition corresponding to $T \setminus \{v_i, i\}$ (given by the branches at $v_i$) and let $P_\partition \subseteq \mathbb{P}^n$ be the corresponding linear space. We have the diagram below:
\begin{equation} \label{eq:psi-boundary-diagram}
\begin{tikzcd}
\Mbar_{0,S} \arrow[r,"\kapi"] & \mathbb{P}^n \\
X_T \arrow[u,hook] \arrow[r,"\kapi"] & \arrow[u,hook,"\iota_\partition"] \mathbb{P}^{k-2}
\end{tikzcd}
\end{equation}
Recall from Equation \eqref{eq:stratum-factors} that $X_T$ is isomorphic to a product of $\Mbar_{0,n'}$'s.
This isomorphism identifies $|\psi_i|$ with the corresponding divisor pulled back from the factor $\Mbar_{0,\deg(v_i)}$, on which one marked point is identified with $i$ and the others correspond canonically to the parts of $\partition$. The bottom horizontal arrow in \eqref{eq:psi-boundary-diagram} is the composition $X_T \to \Mbar_{0,\deg(v_i)} \xrightarrow{|\psi_i|} \mb{P}^{k-2}$.

We calculate directly in projective coordinates.
By Lemma \ref{lem:kap-arbitrary-tree}, $P_\partition$ is given by the equations $z_j = z_k$ whenever $j, k$ are in the same part of $\partition$ and $z_j = 0$ if $j$ is in the same part as $a$. Setting $m$ to be the $i$-minimal marked point of $T$, 
it follows that on $P_\partition$, the Equation \eqref{eq:H-psi} defining $\hyp_i^\psi(t)$ reduces to
  \[0 = t^{m+1} z_{m} + O(t^{m+2})\] if $m<i$, or \[0 = t^{m} z_{m} + O(t^{m+1})\] if $m>i$. In either case, saturating with respect to $t$ and setting $t = 0$ gives the limiting equation $z_{m} = 0$, or simply $y_m = 0$ where $y_m$ indexes the corresponding part of $\partition$.   
  
  The isomorphism \eqref{eq:stratum-factors} identifies the subscheme $\mathbb{V}(y_m)$ with the corresponding $\psi_m$ divisor on the factor $\Mbar_{0,\deg(v_i)}$.  Thus $y_m = 0$ cuts out the reduced union of divisors that have a node (i.e., an edge of the dual graph) separating marked point $i$ from both the marked points $a$ and $m$.  

Back on $X_T$, these divisors correspond to dual graphs $T'$ with a new edge $e$ separating $i$ from $a$ and $m$, such that contracting $e$  results in the original tree $T$ (since $X_{T'}\subseteq X_T$).  By Remark \ref{rmk:slide-contraction}, these are precisely the strata $X_{T'}$ enumerated by $\slide_i(T)$.
\end{proof}

\begin{remark}\label{rmk:simpler-hyperplanes}
In many cases, we can replace $H^\psi_i(t)$ by a simpler equation (by removing some terms) and still get the same result as in Theorem \ref{thm:strata-psi}. In particular, the proof above holds for any hyperplane obtained by deleting entries corresponding to marked points that appear on the branch $A$ of $T$, since those coefficients restrict to $0$ on $X_T$.  

Moreover, if we know the $i$-minimal element $m$ in advance, we can also delete any other summands other than the $x_m$ term in order to slide the $m$ branch towards $a$.  
\end{remark}

\begin{remark}\label{rmk:general-hyperplanes}
 Besides taking subsets of the summands as in the above remark, we can reorder the subscripts on the variables in a hyperplane equation, which results in a modified slide rule. For instance, intersecting $X_T$ with the hyperplane $$z_1+tz_b+t^2z_4+t^3z_2=0$$ applies an $i$-slide in which you look for the branch containing the first among $1,b,4,2$ in that order (so we consider $1$ ``smaller'' than $b$ and so on) and slide that branch away, rather than the $i$-minimal branch as defined above. 
\end{remark}

We now consider arbitrary complete intersections. Recall the following definition from the introduction.

\begin{defHyp}
Let $\mathbf{k}=(k_1,\ldots,k_n)$ be a weak composition.  Let $\vec{t}=(t_{i,j})$ for $1\le i\le n$ and $1\le j \le k_i$ be a tuple of complex parameters.
We denote the subschemes cut out in $\Mbar_{0,n+3}$ by the hyperplanes $\hyp_i^\psi(t_{i,j})$ and $\hyp_i^\omega(t_{i,j})$ as
\begin{align}
V^\psi(\mathbf{k};\vec{t}) = \bigcap_{i=1}^n \bigcap_{j=1}^{k_i} \Psi_n^{-1}(\hyp_i^\psi(t_{i,j})),\\
V^\omega(\mathbf{k};\vec{t}) = \bigcap_{i=1}^n \bigcap_{j=1}^{k_i} \Omega_n^{-1}(\hyp_i^\omega(t_{i,j})),
\end{align}
where $\Psi_n$ is the total Kapranov map and $\emb_n$ is the iterated Kapranov embedding.
\end{defHyp}

\begin{remark}[Monin--Rana's equations for $\Omega_n$]
In order to find the hyperplane equations in Definition \ref{def:all-hyperplanes}, we wrote Mathematica code that used the explicit (conjectural) equations cutting out the embedding $\Omega_n$, due to Leonid Monin and Julie Rana in \cite{MonRan}.  This was an essential tool for experimenting with equations and testing conjectures.
\end{remark}

\begin{example}
For $\mathbf{k}=(1,0,2)$, let $\PP^3\times \PP^3 \times \PP^3$ have coordinates $[x_b:x_c:x_2:x_3]\times [y_b:y_c:y_1:y_3]\times [z_b:z_c:z_1:z_2]$, and let $\PP^1\times \PP^2\times \PP^3$ have coordinates $[x_b:x_c]\times [y_b:y_c:y_1]\times [z_b:z_c:z_1:z_2]$.  Then $V^\psi((1,0,2);\vec{t})$ is defined by the equations \begin{align}
    0&=x_b+t_{1,1}x_c+t_{1,1}^2x_2+t_{1,1}^3x_3, \\ 0&=z_b+t_{3,1}z_c+t_{3,1}^2z_1+t_{3,1}^3z_2, \\ 
    0&=z_b + t_{3,2}z_c+t_{3,2}^2z_1+t_{3,2}^3z_2,
\end{align}
whereas $V^\omega((1,0,2),\vec{t})$ is defined by the equations
\begin{align}
0 &= x_b + t_{1,1}x_c,\\
0 &= z_b + t_{3,1}z_c+t_{3,1}^2z_1+t_{3,1}^3z_2, \\
0 &= z_b + t_{3,2}z_c+t_{3,2}^2z_1+t_{3,2}^3z_2.
\end{align}
\end{example}

\begin{MainThm}
Let $\mathbf{k}$ be a weak composition, and let $\vec{t} = (t_{i, j})$ for $1 \leq i \leq n$ and $1 \leq j \leq k_i$ be complex parameters.
Let $\displaystyle{\lim_{\vec{t} \to \vec{0}}}$ denote the iterated limit
\[\lim_{\vec{t} \to \vec{0}} \big( {-} \big) :=
\lim_{t_{n, k_n} \to 0} \cdots \lim_{t_{n, 1} \to 0} \cdots \cdots
\lim_{t_{2, k_2} \to 0} \cdots \lim_{t_{2, 1} \to 0}\ 
\lim_{t_{1, k_1} \to 0} \cdots \lim_{t_{1, 1} \to 0} \big( {-} \big).\]
(The $i$-th block is empty if $k_i=0$, and $\lim$ denotes the flat limit.) Then we have set-theoretically
\begin{equation}
    \lim_{\vec{t} \to \vec{0}} V^\psi(\mathbf{k};\vec{t}) = \bigcup_{T \in \Slide^\psi(\mathbf{k})} X_T \quad \text{ and } \quad \lim_{\vec{t} \to \vec{0}} V^\omega(\mathbf{k};\vec{t}) = \bigcup_{T \in \Slide^\omega(\mathbf{k})} X_T.
    \end{equation}
Moreover, each boundary stratum $X_T$ appearing in the union is an irreducible component and is generically reduced in the limit.
\end{MainThm}

\begin{proof}
We first consider the $\omega$ case.  We proceed by induction on $n$, then on $\sum k_i$. The case $n=3$ is trivial, as is the case $\sum k_i = 0$.
  
Let $n\ge 3$ and let $\mathbf{k}$ be a weak composition with $\sum k_i\le n$, and assume the statement holds for all smaller $n$ and $\sum k_i$.  Suppose first that $k_n = 0$. Let $\mathbf{k}' = (k_1, \ldots, k_{n-1})$. In this case we have
\[V^\omega(\mathbf{k}; \vec{t}) = \pi_n^{-1} V^\omega(\mathbf{k}', \vec{t}).\]
Flat limits are preserved by flat pullback 
(Lemma \ref{lem:flat-pullbacks}) and $\pi_n$ is flat, so
\[\lim_{\vec{t} \to \vec{0}} V^{\omega}(\mathbf{k}; \vec{t}) = \pi_n^{-1} \Big( \lim_{\vec{t} \to \vec{0}} V^{\omega}(\mathbf{k}'; \vec{t}) \Big).\]
By the induction hypothesis, the right-hand limit is the generically reduced union of boundary strata corresponding to the trees in $\Slide^{\omega}(\mathbf{k}')$.  The preimage $\pi_n^{-1}(X_T)$ of a stratum (with generically reduced scheme structure) is again generically reduced, and is the union of strata $X_{T'}$ formed by inserting the $n$-th marked point into $T$ in all possible ways.  This matches the combinatorial process of the $\omega$-slide algorithm at step $n$ when $k_n=0$, so we obtain the strata corresponding to $\Slide^{\omega}(\mathbf{k})$.

Suppose instead $k_n > 0$. Let $\mathbf{k}'' = (k_1, \ldots, k_n-1)$ and let $\vec{t}''$ denote $\vec{t}$ without $t_{n,k_n}$. By the induction hypothesis, we have
\begin{equation}
    \label{eq:Z}Z = \lim_{\vec{t}'' \to \vec{0}} V^{\omega}(\mathbf{k}''; \vec{t}'') = \bigcup_{T''\in \Slide^{\omega}(\mathbf{k}'')}X_{T''}
\end{equation}
with generically reduced scheme structure on each irreducible component. We now examine the final intersection and limit, and we have 
\begin{align}
\lim_{\vec{t} \to \vec{0}} V^\omega(\mathbf{k}; \vec{t}) &=\lim_{t_{n, k_n} \to 0} \Big( \lim_{\vec{t}'' \to \vec{0}} V^\omega(\mathbf{k}; \vec{t}) \Big) \label{eq:line1} \\ 
&=\lim_{t_{n, k_n} \to 0} \Big( \lim_{\vec{t}'' \to \vec{0}} \Omega_n^{-1}(\hyp_{n}^{\omega}(t_{n, k_n})) \cap V^\omega(\mathbf{k}''; \vec{t}'') \Big) \\
\intertext{Moving all the inner limits inwards then gives}
&\subseteq \lim_{t_{n, k_n} \to 0} \Big( \Omega_{n}^{-1}(\hyp_{n}^{\omega}(t_{n, k_n})) \cap \lim_{\vec{t}'' \to \vec{0}} V^\omega(\mathbf{k}''; \vec{t}'') \Big) \\
&= \lim_{t_{n, k_n} \to 0} \Big( V_n(t_{n,k_n}) \cap Z\Big)\label{eq:line4}
\end{align}
where $V_n(t)=\Omega_n^{-1}(H^\omega_n(t))=|\psi_n|^{-1}(H^{\omega}_n(t))$ as in Theorem \ref{thm:strata-psi} (since the top degree part of the $\Omega_n$ embedding simply agrees with the Kapranov map $|\psi_n|$). We will show that the right-hand side of \eqref{eq:line4} is generically reduced and of the correct dimension. Therefore, by Lemma \ref{lem:generically-equal-limits-iterated}, the left-hand side of \eqref{eq:line1} is also generically reduced and agrees set-theoretically with the right-hand side \eqref{eq:line4}. 

To examine the right-hand side of \eqref{eq:line4}, consider an irreducible component $X_{T''} \subset Z$, where $T'' \in \Slide^\omega(\mathbf{k}'')$ by Equation \eqref{eq:Z}. By Theorem \ref{thm:strata-psi}, 
\[ \lim_{t_{n, k_n} \to 0} V_n(t_{n,k_n}) \cap X_{T''} = \bigcup_{T \in \slide_{n}(T'')} X_{T}\]
with reduced scheme structure.  By Lemma \ref{lem:distinct-insertions} (injectivity of the slide rule), as $T''$ varies, the sets $\slide_n(T'')$ are disjoint, so each resulting stratum $X_T$ occurs exactly once. We thus have set-theoretically
\[\lim_{t_{n, k_n} \to 0} V_n(t_{n,k_n}) \cap Z = \bigcup_{T'' \in \Slide^\omega(\mathbf{k}'')} \bigg(
\bigcup_{T \in \slide_n(T'')}
X_T \bigg) = \bigcup_{T \in \Slide^\omega(\mathbf{k})} X_T,
\]
where each $X_T$ occurs with multiplicity one, i.e. has generically reduced scheme structure, and the last equality is by the definition of the $\omega$-slide rule. This completes the proof for $\omega^\mathbf{k}$.

The argument for $V^\psi(\mathbf{k}; \vec{t})$ and  $\Slide^\psi(\mathbf{k})$ is similar, but takes place entirely in $\Mbar_{0, n+3}$ (without pullbacks). Thus we can, in particular, skip the $k_n=0$ case; let $i$ be largest such that $k_i > 0$. Then the argument is identical to the case $k_n > 0$ for $V^\omega(\mathbf{k}; \vec{t})$, except $\hyp_{n}^\omega(t_{n, k_n})$ is replaced by $\hyp_{i}^\psi(t_{i, k_i})$, and accordingly $\slide_{n}(T'')$ is replaced by $\slide_{i}(T'')$.
\end{proof}

\begin{remark}\label{rmk:flat-families}
It follows from the iterated limit calculation that the parameters $t_{i,j}$ can be replaced, without changing the limit, by powers $t_{i,j} = t^{m_{i,j}}$ of a single parameter $t \to 0$, for some exponents $m_{n,k_n} \gg \cdots \gg m_{n,1} \gg \cdots \gg m_{1,1} \gg 0$. This produces a flat family over $\mathbb{P}^1$.
\end{remark}

As a consequence, we obtain:

\begin{MainCor}
Let $\mathbf{k}$ be a weak composition. Then in $A^\bullet(\Mbar_{0,n+3})$ we have
\begin{equation}\label{eq:maincor}
\psi^\mathbf{k} = \sum_{T \in \Slide^\psi(\mathbf{k})} [X_T], \qquad 
\omega^\mathbf{k} = \sum_{T \in \Slide^\omega(\mathbf{k})} [X_T].
\end{equation}
\end{MainCor}

\begin{example}\label{ex:psi-omega-prod-formula}
By Theorem~\ref{thm:main} and Examples \ref{ex:slide-psi} and \ref{ex:slide-omega}, we have (using the same notation as in those examples) that
\[
\psi_1\psi_3^2 = [X_{T_1}]+[X_{T_2}]+[X_{T_3}] \hspace{1cm}\text{ and }\hspace{1cm} \omega_1\omega_3^2= [X_{T_1}]+[X_{T_3}].
\]
\end{example}

\subsection{Application to \texorpdfstring{$\kappa$}{} classes}

We prove Theorems \ref{thm:kappa} and \ref{thm:generalized-kappa} on kappa classes and generalized kappa classes,
\begin{align*}
\kappa_i &:= (\pi_{n+1})_*(\psi_{n+1}^{i+1}), \\
R_{n; \mathbf{r}} &:= (\pi_{n+1, \ldots, n+m})_*(\psi_{n+1}^{r_1} \cdots \psi_{n+m}^{r_m}).
\end{align*}
We recall the relevant sets of trees:

\begin{itemize}
    \item For $n$ and $i$, the set $K(n, i) \subseteq \Slide^\psi(0^n, i+1)$ consists of the trees $T$ for which $\deg(v_{n+1}) = 3$.
    \item For $n$ and a composition $\mathbf{r} = (r_1, \ldots, r_m)$, the set $R(n;\mathbf{r}) \subseteq \Slide^\psi(0^n, r_1, \ldots, r_m)$ consists of the trees $T$ such that, for each $n+1 \leq j \leq n+m$, the tree $\pi_{j+1, \ldots, n+m}(T)$ has $\deg(v_j) = 3$.
\end{itemize}
We show:

\begin{thm}
 For all $n$ and $i$ and $\mathbf{r}$,
 \[\kappa_i = \sum_{T \in K(n,i)} [X_{\pi_{n+1}(T)}], \qquad R_{n; \mathbf{r}} = \sum_{T \in R(n; \mathbf{r})} [X_{\pi_{n+1, \ldots, n+m}(T)}].\]
\end{thm}
\begin{proof}
By Corollary \ref{cor:main}, we have in $A^\bullet(\Mbar_{0,\{a, b, c, 1, \ldots, n{+}1\}})$
\[\psi_{n+1}^{i+1} = \sum_{T \in \Slide^\psi(0^n, i+1)} [X_T].\]
Pushing forward along $\pi_{n+1}$, we obtain
\[\kappa_i = (\pi_{n+1})_*(\psi_{n+1}^{i+1}) = \sum_{T \in \Slide^\psi(0^n, i+1)} (\pi_{n+1})_*[X_T].\]
Let $T \in \Slide^\psi(0^n, i+1)$ and let $v_{n+1} \in T$ be the internal vertex adjacent to $n+1$. If $\deg(v_{n+1}) > 3$, then $\pi_{n+1}(X_T)$ has dimension lower than $X_T$, so \[(\pi_{n+1})_*[X_T] = 0.\] Otherwise, if $\deg(v) = 3$, then $\pi_{n+1}$ maps $X_T$ isomorphically onto its image $X_{\pi_{n+1}(T)}$, so \[(\pi_{n+1})_*[X_T] = [X_{\pi_{n+1}(T)}].\]
The desired equation for $\kappa_i$ follows. For $R_{n; \mathbf{r}}$, the argument is similar: we apply the pushforward 
\[R_{n; \mathbf{r}} = (\pi_{n+1, \ldots, n+m})_*(\psi_{n+1}^{r_1} \cdots \psi_{n+m}^{r_m})\]
one step at a time, starting from the sum given by the slide set $\Slide^\psi(0^n, \mathbf{r})$. For each $T$, if the degree condition for $R(n; \mathbf{r}) \subseteq \Slide^\psi(0^n, \mathbf{r})$ is satisfied, the pushforward is an isomorphism of $[X_T]$ onto its image. Otherwise, the dimension contracts in some step and the summand vanishes.
\end{proof}

These formulas are not in general multiplicity-free. Indeed, we expect that no multiplicity-free formula can exist for $\kappa_i$ or $R_{n; \mathbf{r}}$ in general; see Problem \ref{prob:kappa}. For $\kappa_i$, we can account for the multiplicities directly.

\begin{corollary}\label{cor:multiplicity}
For all $n$ and $i$, we have
\[\kappa_i = \sum_{T \in \Slide^\psi(0^n, i)} (\deg(v_{n+1})-3)[X_{\pi_{n+1}(T)}].\]
\end{corollary}
\begin{proof}
Let $T \in \Slide^\psi(0^n, i)$. By the calculation above, $T$ contributes to the expression for $\kappa_i$ if, after performing an $(i+1)$st $(n+1)$-slide, the resulting tree $T'$ has $\deg(v_{n+1}) = 3$. That is, the slide should move all but $v_{n+1}$ and exactly one other branch to the new vertex. Since the locations of the $a$ and $m$ branches, and of $v_{n+1}$ itself, are fixed, there are exactly $\deg(v_{n+1})-3$ other choices. Each of these choices has $\pi_{n+1}(T') = \pi_{n+1}(T)$, so $\pi_{n+1}(T)$ arises $\deg(v_{n+1})-3$ times.
\end{proof}

It is not difficult to show that the nonvanishing terms in Corollary \ref{cor:multiplicity} (in which $\deg(v_{n+1})>3$) give a set of \textit{distinct} trees $\pi_{n+1}(T)$.

\begin{example}
We compute $\kappa_1$ on $\Mbar_{0,\{a,b,c,1,2\}}$. We write $(A){-}(B){-}(C)$ to denote the boundary stratum whose dual tree consists of three internal vertices $v_A,v_B,v_C$ along a path, and leaf edges labeled by $A$ (resp.\ $B,C$) attached to $v_A$ (resp.\ $v_B,v_C$). We have, on $\Mbar_{0,\{a,b,c,1,2,3\}}$,
\begin{align*}
    \Slide^\psi(0,0,2) &= \big\{ (ab){-}(c){-}(123),\ (ab){-}(c1){-}(23),\ (ab){-}(c2){-}(13),\ \\ 
    & \qquad (abc){-}(1){-}(23),\ (ab1){-}(c){-}(23),\ (ab2){-}(c){-}(13)\big\}.
\end{align*}
All but the first of these have $\deg(v_3) = 3$. Applying Corollary \ref{cor:multiplicity}, we get $$\kappa_1=2\cdot D(ab|c12)+D(abc|12)+D(ab1|c2)+D(ab2|c1) \in A^1(\Mbar_{0,\{a,b,c,1,2\}}).$$ 
\end{example}

\section{Hyperplanes for lazy tournament points} \label{sec:tournament-hyperplanes}

We now consider the problem of finding parameterized families of hyperplanes whose intersections limit to the sets of points $\Tour(k_1,\ldots,k_n)$ determined by the lazy tournament rule. 

\subsection{Tournaments}

We first recall the definition of lazy tournaments from \cite{GGL-tournaments}.

\begin{definition}\label{def:lazy-tournament}
Let $T$ be a leaf-labeled trivalent tree. The \textbf{lazy tournament} of $T$ is a labeling of the edges of $T$ computed as follows.  Start by labeling each leaf edge (that is, an edge adjacent to a leaf vertex) by the value on the corresponding leaf, as in the second picture of Figure \ref{fig:example-tournament}.  Then iterate the following process:
\begin{enumerate}
    \item \textbf{Identify which pair `face off'.} Among all pairs of labeled edges $(i,j)$ (ordered so that $i<j$) that share a vertex and have a third unlabeled edge $E$ attached to that vertex, choose the pair with the largest value of $i$.
    \item \textbf{Determine the winner.}  The larger number $j$ is the \textit{winner}, and the smaller number $i$ is the \textit{loser} of the match.
    \item \textbf{Determine which of $i$ or $j$ advances.}  Label $E$ by either $i$ or $j$ as follows:
    \begin{enumerate}
        \item If $E$ is adjacent to a labeled edge $u\neq j$ with $u>i$, then label $E$ by $i$.  (We say $i$ \textit{advances}.)
        \item Otherwise, label $E$ by $j$. (We say $j$ \textit{advances}.)
    \end{enumerate}  
\end{enumerate}
We then repeat steps 1-3 until all edges of the tree are labeled. 
\end{definition}

We refer to Step 3(a) above as the \textbf{laziness rule}, since $j$ drops out of the tournament despite winning its match. This happens when $j$ can see that its opponent $i$ will be defeated, again, in its next round against $u$. 

An example of the result of the lazy tournament process is shown in Figure \ref{fig:example-tournament}. 

\begin{figure}
    \centering
    \includegraphics{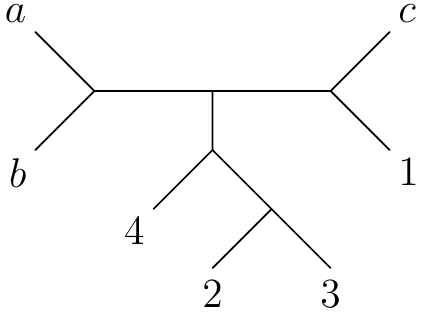}\hspace{1.3cm} \includegraphics{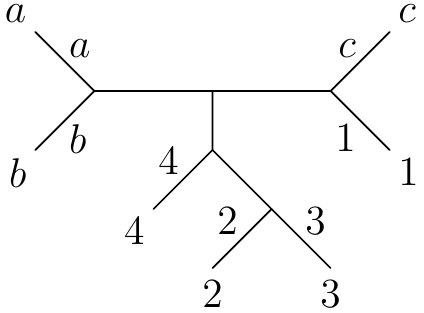}\hspace{1.3cm} \includegraphics{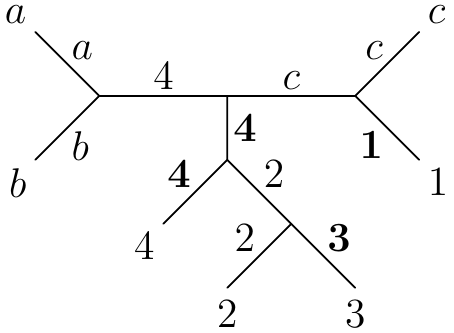}
    \caption{From left to right: A leaf-labeled trivalent tree $T$, its initial labeling of the leaf edges, and its full lazy tournament edge labeling.  Winners of each round of the tournament are shown in boldface at right, indicating $T \in \Tour(1, 0, 1, 2)$.}
    \label{fig:example-tournament}
\end{figure}

\begin{definition}\label{def:tour}
For any weak composition $\mathbf{k} = (k_1,\dots, k_n)$ of $n$, let $\Tour(\mathbf{k})$ be the set of trivalent trees with leaf labels $S$, in which (a) the leaf edges $a$ and $b$ share a vertex, and (b) each label $i\geq 1$ wins exactly $k_i$ times in the tournament.
\end{definition}

In Figure \ref{fig:example-tournament}, the tree $T$ is in $\Tour(1, 0, 1, 2)$.

\begin{thm}[\cite{GGL-tournaments}]\label{thm:tournaments}
 We have $\deg_{\mathbf{k}}(\emb_n)=\int_{\Mbar_{0,S}}\omega_1^{k_1}\cdots \omega_n^{k_n}=|\Tour(\mathbf{k})|.$
\end{thm}
It is therefore natural to ask if we can achieve the tournament boundary points as degenerations of intersections with hyperplanes as well.

From a combinatorial perspective, one advantage of the sets $\Tour(\mathbf{k})$ is that they are disjoint (as $\mathbf{k}$ ranges over all length $n$ compositions of $n$).  This is in contrast to the sets $\Slide(\mathbf{k})$, which all have at least one common tree by Proposition~\ref{prop:common-tree}. Notably, an immediate corollary of Theorem~\ref{thm:tournaments} is that the \textit{total degree} (defined as the sum of the multidegrees) is $$\sum_{\mathbf{k}}\mathrm{deg}_{\mathbf{k}}(\Omega_n)=(2n-1)!!=(2n-1)\cdot (2n-3) \cdot \cdots \cdot 5 \cdot 3 \cdot 1.$$  This enumeration by the odd double factorial follows from the fact that every tree in which $a,b$ is paired occurs in exactly one of the tournament sets (by disjointness), and the trees in which $a,b$ are paired correspond bijectively under $\pi_b$ to the set of all boundary points in $\Mbar_{0,S\setminus b}$.  It is well known that there are $(2n-1)!!$ such points.

\subsection{Hyperplanes for tournaments}

The aim of this section is to prove Theorem \ref{thm:families}, which we restate here for the reader's convenience.

\begin{FamiliesThm}
 Suppose the tuple $\mathbf{k}=(k_1,\ldots,k_n)$ is of one of the following forms:
 \begin{itemize}
     \item $(0,0,\ldots,0,0,n)$,
     \item $(0,0,\ldots,0,1,n-1)$,
     \item $(0,0,\ldots,0,n-1,1)$, or
     \item $(0,0,2,2)$.
 \end{itemize}
Then there exists an explicitly constructed set of hyperplanes in $\PP^1\times \cdots \times \PP^n$, with $k_i$ of them from $\PP^i$ for each $i$, such that their intersection locus $V^{\mathrm{tour}}(\mathbf{k},\vec{t})$ in $\Mbar_{0,n+3}$, pulled back under $\Omega_n$, satisfies 
\begin{equation} \label{eq:limit-tour}
\lim_{\vec{t}\to \vec{0}} V^{\mathrm{tour}}(\mathbf{k};\vec{t})=\Tour(\mathbf{k}).
\end{equation}
Moreover, given any set of hyperplanes satisfying \eqref{eq:limit-tour} for $\mathbf{k} = (k_1,\ldots,k_n)$, there exists such a set for $(k_1,\ldots,k_{n-1},0,k_{n}+1)$.
\end{FamiliesThm}

We prove this in five lemmas; four for the four cases in the theorem, and one for the inductive construction for obtaining $(k_1,\ldots,k_{n-1},0,k_n+1)$ from $\mathbf{k}$. For each one, we construct modified versions of the hyperplanes used in the slide rule, changing which variables appear and in what order. These changes effectively modify the minimality condition in each step of the slide rule; see Remark \ref{rmk:general-hyperplanes}.

Below, we write $[y_b:y_c:y_1:y_2:\cdots:y_{n-2}]$ for the coordinates of $\PP^{n-1}$ and $[z_b:z_c:z_1:z_2:\cdots:z_{n-1}]$ for the coordinates of $\PP^{n}$.

\begin{lemma}
  For $\mathbf{k}=(0,0,\ldots,0,n)$, set $V^{\mathrm{tour}}((0,0,\ldots,0,n);\vec{t})=V^{\omega}((0,0,\ldots,0,n);\vec{t})$. Then $$\lim_{\vec{t}\to \vec{0}} V^{\mathrm{tour}}((0,0,\ldots,0,n);\vec{t})=\Tour(0,0,\ldots,0,n).$$
\end{lemma}

\begin{proof}
It is easily verified, using the slide and tournament rules, that the sets $\Tour(0,0,\ldots,0,n)$ and $\Slide^\omega(0,0,\ldots,0,n)$ coincide.  Indeed they both only contain the single tree:  
\begin{center}
    \includegraphics{Figures/first-lemma.pdf}
\end{center}
(see Proposition \ref{prop:common-tree}).  Thus we are done by Theorem \ref{thm:main}.
\end{proof}

Throughout the remainder of this section, we will say that $V^{\mathrm{tour}}(\mathbf{k};\vec{t})$ is \textit{defined by} a given set of hyperplane equations in $\PP^1 \times \PP^2 \times \cdots \times \PP^n$ if it is equal to $\Omega_n^{-1}$ of the vanishing locus of those equations.

\begin{lemma}
Define $V^{\mathrm{tour}}((0,0,\ldots,0,1,n-1);\vec{t})$ by the set of equations: $$y_b=0, \hspace{0.5cm} z_b=t_2 z_{n-1}, \hspace{0.5cm}z_c=t_3z_1, \hspace{0.5cm}z_1=t_4z_2, \hspace{0.5cm}z_2=t_5z_3,\hspace{0.5cm} \ldots,\hspace{0.5cm} z_{n-3}=t_{n}z_{n-2}$$ where $\vec{t}=(t_1,t_2,\ldots,t_n)$.  Then $$\lim_{\vec{t}\to \vec{0}} V^{\mathrm{tour}}((0,0,\ldots,0,1,n-1);\vec{t})=\Tour(0,0,\ldots,0,1,n-1).$$
\end{lemma}

\begin{proof}
  Intersecting with the first equation, $y_b=0$, restricts to the divisors in which $b$ is on the $a$ branch from the perspective of $n-1$.  Moreover, since we will be intersecting with $n-1$ hyperplanes in $\PP^n$, we may restrict our attention to divisors in which $n$'s internal vertex $v_n$ has degree at least $n+2$.  In particular, we may restrict to the boundary strata
  $$D(\{a,b\}|\{c,1,2,\ldots,n-1,n\})\cup \bigcup_{j\in \{c,1,2,\ldots,n-2\}} D(\{a,b,c,1,\ldots,\hat{j},\ldots,n-2,n\}|\{j,n-1\}).$$
  
  First consider the divisor $D(\{a,b\}|c,1,2,\ldots,n)$.  Then by Remark \ref{rmk:general-hyperplanes}, intersecting with $z_b=t_2z_{n-1}$ and taking the limit as $t_2\to 0$ effectively sets $z_{n-1}=0$, which treats $n-1$ as the minimal element and slides it towards $a$.  We can again restrict by dimensionality to the stratum in which the three internal vertices have leaves $\{a,b\}$, $\{n-1\}$, and $\{c,1,2,\ldots,n-2,n\}$.  The remaining equations similarly slide $c,1,\ldots,n-3$ towards $a$, yielding the unique point shown below.  
  \begin{center}
      \includegraphics{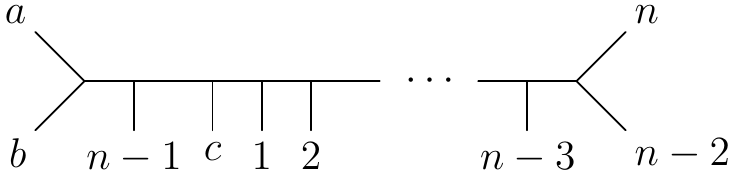}
  \end{center}
  Now consider a divisor of the form $D(\{a,b,c,1,\ldots,\hat{j},\ldots,n-2,n\}|\{j,n-1\})$.  The first equation, $z_b=t_2z_{n-1}$, simply says that we slide $b$ towards $a$ (so that they share an internal vertex), and again by dimensionality we can restrict to the case in which all remaining edges are still attached to the same internal vertex as $n$.  The remaining equations similarly slide $c,1,2,\ldots,j-1$ in that order towards $a$, then move the branch containing the pair $j,n-1$ towards $a$, and finally move $j+1,\ldots,n-2$ towards $a$.  An example is shown below for $n=6$ and $j=2$.
  \begin{center}
      \includegraphics{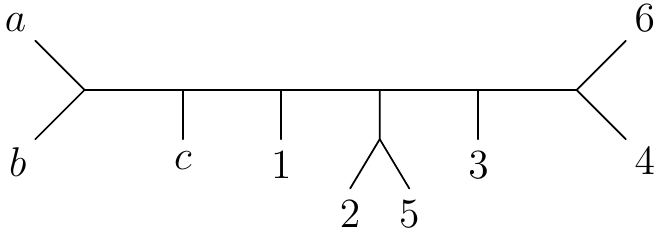}
  \end{center}
  
  One can easily verify that these are precisely the boundary points whose lazy tournament has $n-1$ winning one round and $n$ winning the rest.
\end{proof}

\begin{lemma}
Define $V^{\mathrm{tour}}((0,0,\ldots,0,n-1,1);\vec{t})$ by the set of $n-1$ equations defining the smaller locus $V^{\omega}((0,0,\ldots,0,n-1);\vec{t})$ in the $y$ variables, plus the single equation $$z_b=t_n z_{n-1}$$ in the $z$ variables.  Then $$\lim_{\vec{t}\to \vec{0}} V^{\mathrm{tour}}((0,0,\ldots,0,n-1,1);\vec{t})=\Tour(0,0,\ldots,0,n-1,1).$$
\end{lemma}

\begin{proof}
Intersecting with the first $n-1$ equations and taking the corresponding limits, we know for size $n-1$ we obtain the unique tree $T_0$ in $\Slide^\omega(0,0,\ldots,0,n-1)$, namely the caterpillar tree $T_0$ with $a,b$ on one end, $n-2,n-1$ on the other, and leaves $c,1,2,\ldots,n-3$ in order in between.  Thus on $\Mbar_{0,n+3}$ we are in the union of divisors in $\pi_{n}^{-1}(T_0)$, given by inserting the leaf $n$ to attach to any one of the internal vertices of $T_0$.

  We now consider the equation $z_b=t_nz_{n-1}$.  Intersecting and taking the limit with a divisor in which $n$ and $b$ are on the same vertex slides the $b$ towards $a$, and otherwise slides the branch containing $n-1$ towards $a$.  In the former case we get the point:
\begin{center}
    \includegraphics{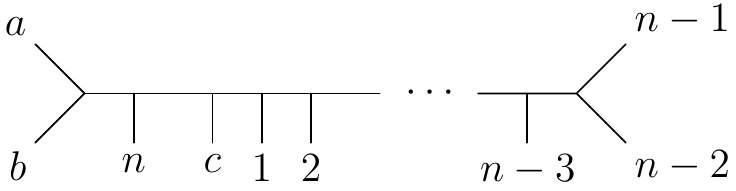}
\end{center}
 and in the latter cases we get points that look like (for $n=6$, where the $6$ may be merged with any of the other points $c,1,3,4,5$ rather than with $2$): 

\begin{center}
    \includegraphics{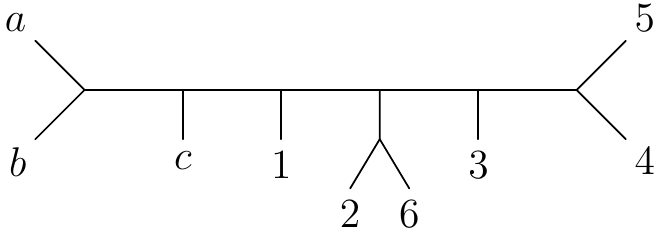}
\end{center}

These are precisely the trees whose lazy tournament has $n-1$ winning $n-1$ rounds and $n$ winning once.
\end{proof}

\begin{lemma}
Define $V^{\mathrm{tour}}((0,0,2,2);(t_1,t_2,t_3,t_4))$ by the set of equations $$y_b=0,\hspace{0.5cm}y_c+t_2 y_1+t_2^2y_2=0, \hspace{0.5cm}z_b+t_3 z_3=0, \hspace{0.5cm}z_c+t_4z_1+t_4^2z_2=0.$$  Then $$\lim_{\vec{t}\to \vec{0}} V^{\mathrm{tour}}((0,0,2,2);\vec{t})=\Tour(0,0,2,2).$$
\end{lemma}

\begin{proof}
  Since these equations are for one single multidegree, we have simply verified via a computer computation that the intersections limit to the six lazy tournament points in $\Tour(0,0,2,2)$. 
  
  For completeness we also provide a brief proof along the lines of the previous lemmas.  The first equation indicates that $a,b$ are separated from $3$ in the tree in $\Mbar_{0,\{a,b,c,1,2,3\}}$, and the second peforms a $3$-slide where the possible minimal elements are $c,1,2$ in that order.  Writing $(A){-}(B){-}(C)$ to denote the boundary stratum given by the tree with three internal vertices along a path whose leaves are labeled by the sets $A,B,C$ in that order, it follows that we are on one of the (inverse images under $\pi_4$ of the) boundary strata \begin{center}
      $(ab){-}(c){-}(123)$,\,\,$(ab){-}(c1){-}(23)$,\,\,$(ab){-}(c2){-}(13)$, \\  $(abc){-}(1){-}(23)$,\, $(ab1){-}(c){-}(23)$, \,$(ab2){-}(c){-}(13)$
  \end{center}
       in $\Mbar_{0,\{a,b,c,1,2,3\}}$. 
  Pulling back under $\pi_4$, we insert $4$ at a leaf, and by dimensionality we may restrict to the case in which $4$ is inserted at the vertex of degree $4$ in each case above.  The equation $z_b+t_3z_3=0$ slides either the branch containing $b$ (from $4$'s perspective) towards $a$ if the $b$ and $a$ branch do not coincide, and otherwise slides the branch containing $3$ towards $a$.  The final equation then performs an ordinary $4$-slide.  This degeneration process yields $6$ points in Figure \ref{fig:0022}, which are precisely the points of $\Tour(0,0,2,2)$.
\end{proof}

\begin{figure}
    \centering
      \includegraphics{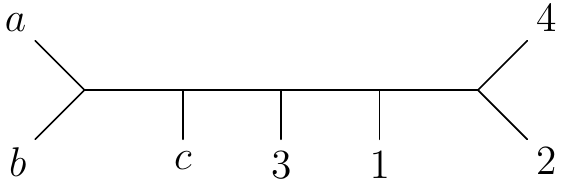} \hspace{1cm} \includegraphics{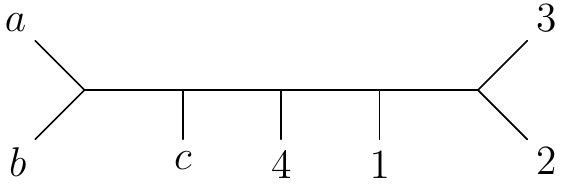} \\\vspace{1cm}
      \includegraphics{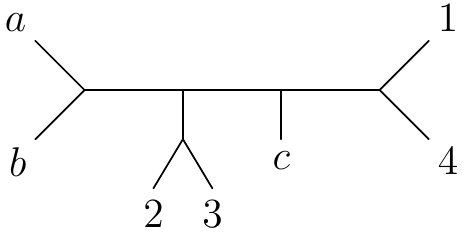}
      \hspace{2cm}
      \includegraphics{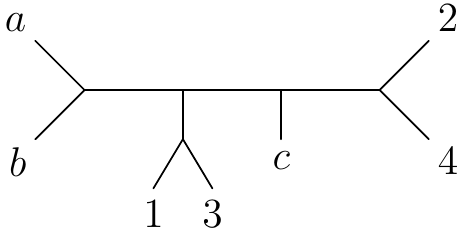} \\\vspace{1cm}
      \includegraphics{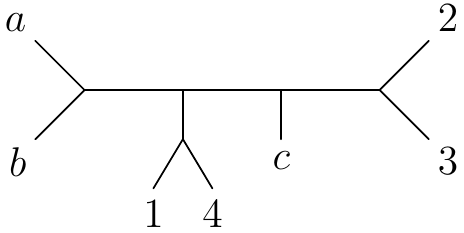}
      \hspace{2cm}
      \includegraphics{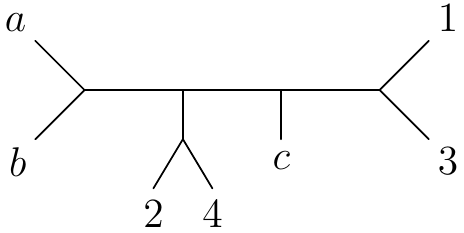} 
    \caption{The six points in $\Tour(0,0,2,2)$.}
    \label{fig:0022}
\end{figure}

The final lemma below completes the proof of Theorem \ref{thm:families}.  We still use $z$ variables to label $\PP^n$ below, and now use $w$ variables to label $\PP^{n+1}$.

\begin{lemma}
Let $\mathbf{k}$ be a composition of $n$ for which $V^{\mathrm{tour}}(\mathbf{k};\vec{t})$ is already defined.  Define $$V^{\mathrm{tour}}((k_1,\ldots,k_{n-1},0,k_{n}+1);\vec{t})$$ by changing the variables $z_i$ of the last $k_n$ equations defining $V^{\mathrm{tour}}(\mathbf{k};\vec{t})$ to the variables $w_i$ of $\PP^{n+1}$, and also adding the additional equation $$w_b+t_{n+1}w_c+t_{n+1}^2w_1+t_{n+1}^3w_2+\cdots+t_{n+1}^{n}w_{n-1}=0.$$  Then $$\lim_{\vec{t}\to \vec{0}} V^{\mathrm{tour}}((k_1,\ldots,k_{n-1},0,k_n+1);\vec{t})=\Tour(k_1,\ldots,k_{n-1},0,k_n+1).$$
\end{lemma}

\begin{proof}
  First note that the tournament points of $\Tour(k_1,\ldots,k_{n-1},0,k_n+1)$ are in bijection with those of $\Tour(k_1,\ldots,k_n)$, and can be formed from the smaller trees by inserting $n+1$ to pair with $n$.  We show that the process of twisting up the existing hyperplanes and adding the new hyperplane equation has this exact same effect on the intersection points.
  
  Indeed, the equations in all $\PP^i$ for $i\le n-1$ give the same strata as before, and then we pull back under $\pi_n$ and $\pi_{n+1}$ by inserting $n$ and $n+1$ in all possible ways.  Then, applying the relabeled equations in $\PP^{n+1}$ coming from the ones we had before in $\PP^n$ apply the same slide moves except from the perspective of $n+1$ instead of $n$ (ignoring the position of $n$).  But then we need to do a final intersection at $n+1$, so in fact the leaf $n$ must remain attached to $n+1$ at each step.  The final equation then does an ordinary $n+1$-slide, which means that $n$ (being non-minimal) stays attached to $n+1$ and the other branch slides towards $a$. This process is equivalent to making the $n+1$ and $n$ leaves collide.  This completes the proof.
\end{proof}

\section{Further discussion and open problems}\label{sec:conclusion}

We conclude with some further observations and avenues for future research, both in combinatorial directions (Sections \ref{sec:comb-first} through \ref{sec:comb-last}) and geometric (Sections \ref{sec:geom-first} through \ref{sec:geom-last}).

\subsection{Tournaments vs slide points}\label{sec:comb-first}

It follows from~\cite[Theorem 1.5]{GGL-tournaments} and Corollary~\ref{cor:multidegrees} that $|\Tour(\mathbf{k})| = \deg_{\mathbf{k}}(\Omega_n) = |\Slide^\omega(\mathbf{k})|$. These two identities were obtained using different methods. The first follows from a bijection with \emph{column-restricted parking functions} \cite{CGM, GGL-tournaments} which naturally satisfy the asymmetric string recursion. The second follows from counting intersection points with parametrized hyperplanes, and has the inductive structure of the slide rule.

\begin{problem}\label{prob:bijection}
Find a combinatorial bijection between the sets $\Tour(\mathbf{k})$ and $\Slide^\omega(\mathbf{k})$.
\end{problem}

One possible route to solving this problem is to use column-restricted parking functions as an intermediate object.  Along these lines, for the $\Psi_n$ setting, parking functions may be generalized to a set of objects enumerated by the ordinary multinomial coefficient $$\binom{n}{k_1,\ldots,k_n}=\frac{n!}{k_1!\cdots k_n!}=\deg_\mathbf{k}(\Psi_n)$$ (when $\sum k_i=n$).  We sketch here one way to see combinatorially that $|\Slide^\psi(\mathbf{k})|=\binom{n}{k_1,\ldots,k_n}$ for $\sum k_i=n$. We assign to each tree $T$ in $\Slide^\psi(\mathbf{k})$ a word $w$ in the letters $1,2,\ldots,n$ in which the letter $i$ occurs $k_i$ times. We construct $w$ by beginning with an empty word; then at each $i$-slide, we insert an $i$ into $w$ as follows. For each internal vertex $v \in T$, let $j_v$ be the minimal leaf vertex among the non-$a$ branches of $T$ at $v$. Order the internal vertices $v$ by the value of $j_v$, breaking ties by saying $v > v'$ if $v$ is closer to $a$. Let $v_i$ be the internal vertex adjacent to leaf $i$, and let $j$ be the position of $v_i$ in the ordering of the internal vertices. Then we insert $i$ into $w$ at the $j$th position from the left.

This suggests the possibility of constructing an analogous bijection between $\Slide^{\omega}(\mathbf{k})$ and the column-restricted parking functions, which in turn are in bijection with $\Tour(\mathbf{k})$.

\subsection{Pattern avoidance}

One difficulty in Problem \ref{prob:bijection} is that the sets $\Slide^\omega(\mathbf{k})$ and $\Tour(\mathbf{k})$ do not always consist of trees of the same shapes.  For instance, when $\mathbf{k}=(1,1,\ldots,1)$, every element of $\Tour(\mathbf{k})$ corresponds to a {\bf caterpillar graph}, meaning that its internal vertices form a path. Not every element of $\Slide^\omega(\mathbf{k})$, however, is a caterpillar. Intriguingly, there is a characterization of the caterpillar graphs in $\Slide^\omega(\mathbf{k})$ via permutation pattern avoidance.

We say a permutation $\pi$ {\bf avoids the pattern $23$-$1$} if there do not exist indices $i$ and $j$ with $i+1<j$ such that $\pi_j < \pi_i<\pi_{i+1}$. For example, the $15$ permutations on $4$ letters that avoid $23$-$1$ are
\[4321,3214, 4213, 2143, 2134, 4312, 3142, 3124, 4132, 1432, 1324, 4123, 1423, 1243, 1234,
\]
whereas the permutation $2431$ contains a $23$-$1$ pattern with $i=1,j=4$. It turns out that the slide labelings on caterpillar graphs in $\Slide^{\omega}(1,1,\ldots,1)$ correspond precisely to the $23$-$1$-avoiding permutations.  For instance, the following tree occurs in $\Slide^\omega(1,1,1,1)$ and has a slide labeling whose labeled internal edges, from left to right, form the word $2143$: 

\begin{center}
    \includegraphics{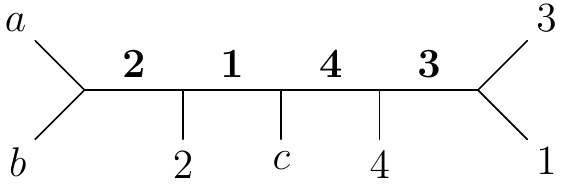}
\end{center}

It would be interesting, and might shed new light on the structure of $\Slide^\omega(\mathbf{k})$, to describe the set (or various subsets of it) by pattern avoidance conditions. Notably, this may be an avenue through which to recover the asymmetric string recursion, and so obtain a bijection to tournaments.

We prove this general correspondence between caterpillar graphs in $\Slide^\omega(1,1,\dots, 1)$ and $23$-$1$-avoiding permutations here.  Below, we use the convention that the leaves $a,b$ are drawn on the left and the path moves out towards the right, so moving left (resp. right) means moving along the path towards (resp. away from) $a$.

\begin{prop}\label{prop:patterns}
Let $\mathrm{Cat}^\omega_n \subseteq \Slide^\omega(1,1,\dots,1)$ be the subset of trivalent trees that correspond to caterpillar curves. For each tree $T\in \mathrm{Cat}^\omega_n$, define the word $w(T)$ by reading the labels in the slide labeling of $T$ from left to right. The set of words $$\{w(T):T\in \mathrm{Cat}^\omega_n\}$$ are precisely the $23$-$1$-avoiding permutations of length $n$, and in fact the words $w(T)$ are all distinct. 
\end{prop}

To prove this, we define the following \textbf{leaf labeling algorithm}.

\begin{definition}[Leaf labeling algorithm]
Let $w$ be a $23$-$1$-avoiding permutation.  Define the tree $T_w$ to be the tree constructed as follows: First label the internal edges of a caterpillar tree by $w_1,\ldots,w_n$ from left to right, and label the leftmost two leaves $a,b$.  Then label the remaining leaves $n,n-1,n-2,\ldots,1,c$ in descending order via the following rule:  
    
At step $n-i$, let ${\bf j}$ be the edge label just to the right of edge ${\bf n-i}$ (if such an edge ${\bf j}$ exists).  \\
    \textbf{Case 1:} If ${\bf j}<{\bf n-i}$, then label the leaf just to the right of ${\bf n-i}$ by $n-i$.\\
    \textbf{Case 2:} If ${\bf j}>{\bf n-i}$ or ${\bf j}$ does not exist, label the rightmost unlabeled leaf to the right of ${\bf n-i}$ by $n-i$.

Finally, label the remaining unlabeled leaf by $c$.
\end{definition}

\begin{remark}\label{rmk:larger-right}
At any Case 2 step, all edge labels to the right of ${\bf n-i}$ are greater than ${\bf n-i}$, for otherwise ${\bf n-i}$ and $\bf j$ would form a $23$-$1$ pattern with a smaller label to the right.
\end{remark}

As an example, the tree shown above for the permutation $w=2,1,4,3$ is precisely the tree $T_w$ obtained by the leaf labeling algorithm.  The following lemma shows that the algorithm is always well-defined.

\begin{lemma}\label{lem:well-defined}
Whenever Case 2 of the leaf labeling algorithm applies, there are exactly two unlabeled leaves available to the right of edge ${\bf n-i}$, one of which is the leaf just to the right of it.  Whenever Case 1 applies, the leaf just to the right of ${\bf n-i}$ has not yet been labeled.
\end{lemma}

\begin{proof}
For the Case 2 claim, we first show that at step $n-i$, the only leaves to the right of edge ${\bf n-i}$ that have already been labeled are labeled by the edge values to the right of ${\bf n-i}$.  Assume for contradiction that some leaf to the right of ${\bf n-i}$ is labeled by $y>n-i$ where $\bf y$ is to the left of ${\bf n-i}$.  Then since leaf $y$ is not adjacent to edge $\bf y$, it was labeled using Case 2 on step $\bf y$, and so in fact ${\bf n-i}>{\bf y}$ by Remark \ref{rmk:larger-right}, a contradiction.

Let $k$ be the number of internal edges to the right of ${\bf n-i}$; then there are $k+2$ leaves to the right of ${\bf n-i}$, and so at least two leaves to the right of ${\bf n-i}$ are available. By induction on $i$, we may assume the earlier steps of the algorithm are well-defined, in particular each leaf $x > n-i$ is to the right of the edge labeled ${\bf x}$. This shows that the leaf to the right of the edge ${\bf n-i}$ is unlabeled; and there is exactly one other unlabeled edge further to the right.
    
For Case 1, suppose for contradiction that the leaf just to the right of ${\bf n-i}$ was already labeled on a previous step, say by $m >n-i$.  Then on step $m$, since $m$ is not just to the right of edge label $\bf m$, it used Case $2$ of the algorithm.  Thus edge label $\bf m$ is just to the left of some ${\bf j'}>\bf m$, and both are to the left of ${\bf n-i}$.  Note that $j'>n-i$, so $\bf m,{\bf j'},{\bf n-i}$ form a $23$-$1$ pattern, a contradiction.
\end{proof}

\begin{proof}[Proof of Proposition \ref{prop:patterns}]
First note that the words $w(T)$, which come from the slide labeling, are distinct since they are constructed inductively by starting with $1$ and then inserting a $2$, $3$, $4$, etc, with the position of insertion corresponding to the position we insert the new leaf at the $i$-th step of the slide rule.  

We next show by induction on $n$ that each of the words $w(T)$ is $23$-$1$ avoiding.  Assume it is true for $n-1$, and let $T\in \mathrm{Cat}^\omega_n$.  Then deleting the leaf $n$ from $T$ results in a caterpillar tree $S=\pi_n(T)\in \mathrm{Cat}^{\omega}_{n-1}$, so the slide labeling of $S$ is $23$-$1$ avoiding by the inductive hypothesis.  
  
Note that in the slide labeling of $T$, the internal edge just left of leaf edge $n$ is labeled first, by $\bf n$, and then the remaining edges are labeled as they were in $S$.  Therefore, the word $w(T)$ is obtained by inserting $\bf n$ into $w(S)$ accordingly. So, to show that $w(T)$ is still $23$-$1$ avoiding, it suffices to show that the $\bf n$ that is inserted does not create a $23$-$1$ pattern.  Let $\bf x$ be the slide label just left of $\bf n$ in $w(T)$, and assume for contradiction that there is some slide label $\bf y<\bf x$ to the right of $\bf n$ in $w(T)$.  Let $z$ be the leaf just to the left of the slide label $\bf n$; then by the definition of the $\omega$-slide labeling, $z$ is less than all leaf labels to its right.  Thus in particular $z<y$ and so $z<x$ by transitivity.  
  
  \begin{center}
     \includegraphics{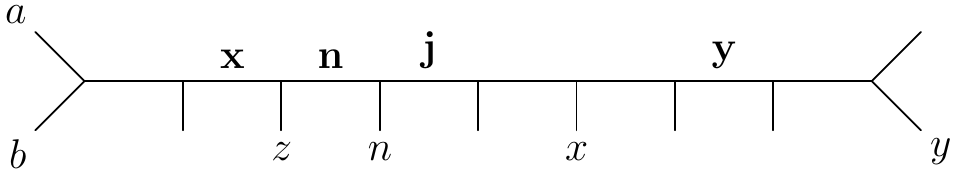}
  \end{center}
  
In particular, $z\neq x$, so $x$ labels some leaf to the right of $n$.  Then since the slide label left of $\bf n$ is $\bf x$, the internal edge labels on the path from leaf $n$ to $x$ must all be greater than $\bf x$ as well; let $\bf j$ be the leftmost such label.  Then $\bf y< \bf x<\bf j$ and these three edges form a $23$-$1$ pattern in $w(S)$, a contradiction.  It follows that $w(T)$ is $23$-$1$ avoiding as well.

We finally show that if $w$ is any $23$-$1$-avoiding permutation, the tree $T_w$ obtained by the leaf labeling algorithm has valid slide labeling $w$.  It suffices to check the condition (3) in Definition \ref{def:slide-labeling} comparing minimal elements.  We first check the condition at the edge label $\bf n$.  Let $z$ label the leaf just left of $\bf n$, and let $\bf x$ be the edge label just left of $\bf n$.  At step $x$ of the leaf labeling algorithm, since $n>x$ we are in Case 2, and so the leaf labeled by $x$ is to the right of $\bf n$ by Lemma \ref{lem:well-defined}.  Moreover, all other leaves to the right of $\bf n$ were already labeled and are greater than $x$.  Thus $x$ is the minimal leaf label to the right of $\bf n$.  Furthermore, since the labeling of leaf $z$ occurs after $x$, we have $z<x$.  Therefore the slide labeling is valid at $\bf n$. It is valid for all smaller labels by a similar argument after contracting edge $\bf n$ and deleting leaf $n$ (since $n$ labels the leaf just after $\bf n$).
\end{proof}

Since the number of $23$-$1$-avoiding permutations is the $n$th Bell number $B_n$ (see Claesson \cite{Claesson} and OEIS entry A000110 \cite{oeis}), we therefore have the following corollary.

\begin{corollary}
The number of caterpillars in $\Slide^\omega(1,1,\dots,1)$ is the $n$th Bell number $B_n$.
\end{corollary}

\subsection{The \texorpdfstring{$S_n$}{} action and slide sets}\label{sec:comb-last}

The symmetric group $S_n$ acts on $\Mbar_{0,n+3}$ by permuting the marked points $1,\ldots,n$. Likewise, it acts on psi classes and boundary strata by relabeling. Thus, permuting the leaves of the trees in $\Slide^{\psi}(k_1,\ldots,k_n)$ according to a permutation $\sigma \in S_n$ gives a positive formula for the product 
\[\psi_{\sigma(1)}^{k_1} \cdots  \psi_{\sigma(n)}^{k_n} = \psi_1^{k_{\sigma^{-1}(1)}} \cdots  \psi_n^{k_{\sigma^{-1}(n)}}\]
as the sum of boundary classes $[X_{\sigma(T)}]$ for $T\in \Slide^{\psi}(\bf k)$. These strata may be obtained as the limiting intersections with the hyperplanes formed by applying $\sigma$ to each of the hyperplanes defining $V^{\psi}(\mathbf{k},\vec{t})$ (this also has the effect of changing a hyperplane of class $\psi_i$ to one of class $\psi_{\sigma(i)}$ and relabeling the projective coordinates).  However, this gives a different set of trees than those enumerated by $\Slide^{\psi}(k_{\sigma^{-1}(1)},\ldots,k_{\sigma^{-1}(n)})$, because the slide rule is sensitive to the ordering of the indices, and the iterated limit is also effectively taken in a different order.

Nonetheless, the two resulting sets of strata must be equinumerous (see Remark \ref{rmk:multiplicity}). Therefore, there must be a bijection between $\Slide^\psi(k_1,\ldots,k_n)$ and $\Slide^\psi(k_{\sigma(1)},\ldots,k_{\sigma(n)})$.

\begin{problem}
For any permutation $\sigma \in S_n$ and any composition $\mathbf{k}$, construct a combinatorial bijection between $\Slide^\psi(k_1,\ldots,k_n)$ and  $\Slide^\psi(k_{\sigma(1)},\ldots,k_{\sigma(n)})$.
\end{problem}

As discussed above, the bijection itself is not given by simply applying a permutation to the leaf labels of the trees. In fact, even the shapes of the trees are not preserved; the shapes in $\Slide^{\psi}(0,1,2)$ do not match those of $\Slide^\psi(0,2,1)$.

This problem boils down to understanding how reordering the indices on the hyperplane equations changes the slide points that we obtain.  For a single $i$-slide, it simply changes the notion of the ``$i$-minimal element''. After more than one slide, however, the resulting trees may be very different.

A slight variant is to consider arbitrary sequences of slides, such as $\psi_1\psi_2\psi_1\psi_2$:

\begin{problem}
Let $w = w_1 \cdots w_c$ be a word in the symbols $1, \ldots, n$, containing $k_i$ $i$'s for each $i$. Let $\Slide^\psi_{\mathrm{word}}(w)$ denote the set of trees obtained by performing a $w_1$-slide, then a $w_2$-slide, and so on. Construct a combinatorial bijection between $\Slide^\psi_{\mathrm{word}}(w)$ and $\Slide^\psi(k_1, \ldots, k_n).$
\end{problem}

\subsection{Limiting hyperplanes for tournament points (general case)}\label{sec:geom-first}

In Section~\ref{sec:tournament-hyperplanes}, we exhibit certain infinite families of tournament points as limiting hyperplane intersection points. It remains to be seen whether all tournament points admit such a geometric realization. A hint toward achieving this goal is~\cite[Theorem 1.8]{GGL-tournaments}, which states that the coordinates of the points $\Tour(k_1,\dots, k_n)$ in the $\mathbb{P}^r$ factor all lie on the $k_r$ hyperplanes
\[ z_b = 0, z_c = 0,z_1 = 0,\dots, z_{k_r-2} = 0\]
where $[z_b:z_c:z_1:\cdots:z_{r-1}]$ are the projective coordinates of $\mathbb{P}^r$. This suggests looking for a parametrized family of hyperplanes such that the hyperplanes themselves limit to the ones listed above. The smallest case not covered by the results in Section~\ref{sec:tournament-hyperplanes} is $\mathbf{k} = (1,1,1)$.

\begin{problem}\label{prob:tournament-hyperplanes}
Generalize Theorem~\ref{thm:families} to all Catalan tuples $(k_1,k_2,\dots, k_n)$.
\end{problem}

For $\mathbf{k}=(1,1,1)$, we could not find an appropriate family of hyperplanes using modified slides as in Section \ref{sec:tournament-hyperplanes}. We suspect that it is not possible. It may instead be necessary to modify the tournament points themselves (for example, the position of the leaf $b$ is mostly irrelevant to the tournament algorithm). 

\subsection{Reducedness}

We have seen that the limiting intersections in Theorem \ref{thm:main} are generically reduced. 

\begin{problem}\label{prob:reducedness}
Determine whether the limiting fibers in Theorem \ref{thm:main} are reduced.
\end{problem}

We do not know the answer to this question when $\sum k_i < n$. An affirmative answer would mean that Theorem \ref{thm:main} also computes $\psi^\mathbf{k}$ and $\omega^\mathbf{k}$ in the  $K$-theory ring $K(\Mbar_{0,n+3})$, as the class of the structure sheaf of a union of strata. If so, and if the components $X_T$ for $T \in \Slide^\psi(\mathbf{k})$ intersect sufficiently nicely, it would be possible to extract K-theoretic formulas for $\psi^\mathbf{k}$ and $\omega^\mathbf{k}$ as alternating sums in the classes of the structure sheaves $[\mathcal{O}_{X_T}]$, by inclusion-exclusion.

\subsection{Kappa classes and multiplicity}\label{sec:kappa-questions}

Our formulas for kappa classes and generalized kappa classes, Theorem \ref{thm:kappa} and Corollary \ref{cor:multiplicity}, consist of boundary classes with multiplicities often greater than $1$. In general, we expect that no multiplicity-free formula can exist.

\begin{problem} \label{prob:kappa}
Fix $\mathbf{r} = (r_1, \ldots, r_m)$. Let $c = \sum r_i - m$ and let $s_n^c$ be the number of boundary strata of codimension $c$ on $\Mbar_{0,n+3}$. Is it true that \[\lim_{n \to \infty} \frac{|R(n; \mathbf{r})|}{s_n^c} = \infty\ ?\]
\end{problem}
Indeed, $\kappa_0$ is $n+1$ times the fundamental class of $\Mbar_{0.n+3}$. For $\kappa_1$, a straightforward summation in Corollary \ref{cor:multiplicity} shows that $\kappa_1$ is the sum of $(n-1) 2^n + 1$ boundary divisors (counted with multiplicity), whereas $\Mbar_{0,n+3}$ has only $4 \cdot 2^n - n - 4$ distinct boundary divisors. Hence, by Remark \ref{rmk:multiplicity}, $\kappa_1$ can't be expressed as a multiplicity-free sum of boundary divisors for $n>5$, and the limit in Problem \ref{prob:kappa} holds.

\subsection{Other intersection products}\label{sec:geom-last}

Finally, it would be interesting to extend the methods of this paper to other intersection products on moduli spaces of curves.

\begin{problem}
Construct degenerations of complete intersections of $\psi$ and $\omega$ classes on Hassett spaces $\Mbar_{0,\vec{w}}$ \cite{hassett2003}.
\end{problem}

We expect that the methods of this paper are special to genus $0$, but any extensions to positive genus would also be of interest.

\bibliography{myrefs}
\bibliographystyle{plain}

\end{document}